\documentclass[11pt]{amsart}
\usepackage{fullpage}
\usepackage{amsfonts}
\usepackage{amssymb}
\usepackage[T1]{fontenc}
\usepackage[utf8]{inputenc}
\usepackage[english]{babel}
\usepackage{adjustbox}
\usepackage{mathabx}
\usepackage{hyperref}
\usepackage{framed}
\usepackage{csquotes}
\usepackage{times}
\usepackage{amsmath}
\usepackage{url}
\usepackage{blindtext}
\usepackage{tikz-cd}
\usepackage{mathtools}
\usepackage{graphicx}
\usepackage{amsmath,calligra,mathrsfs}
\setquotestyle{english}
\usepackage[backend=bibtex,
style=alphabetic,
bibencoding=ascii
]{biblatex}
\newtheorem{thm}{Theorem}[section]
\newtheorem{cor}[thm]{Corollary}
\newtheorem{lem}[thm]{Lemma}
\newtheorem{prop}[thm]{Proposition}
\newtheorem{claim}[thm]{Claim}

\theoremstyle{definition}
\newtheorem{rem}[thm]{Remark}
\newtheorem{ex}[thm]{Example}
\newtheorem{defn}[thm]{Definition}

\newcommand{\mc}{\mathcal}
\newcommand{\mb}{\mathbf}

\newcommand{\Z}{\mathbb{Z}}
\newcommand{\R}{\mathbb{R}}
\newcommand{\C}{\mathbb{C}}
\newcommand{\Q}{\mathbb{Q}}
\newcommand{\ov}{\overline}

\newcommand{\mf}{\mathfrak}
\newcommand{\N}{\mathbb{N}}
\newcommand{\Vect}{\mathrm{Vect}}
\newcommand{\id}{\mathrm{id}}
\newcommand{\Res}{\mathrm{Res}}

\newcommand{\RigIsoc}{\mathrm{RigIsoc}}
\newcommand{\GL}{\mathrm{GL}}
\DeclareMathOperator{\Kal}{Kal}
\DeclareMathOperator{\Ob}{obj}

\DeclareMathOperator{\Int}{Int}
\DeclareMathOperator{\bas}{bas}
\DeclareMathOperator{\msc}{sc}
\DeclareMathOperator{\tor}{tor}
\DeclareMathOperator{\Hom}{Hom}
\DeclareMathOperator{\ad}{ad}
\DeclareMathOperator{\WD}{WD}
\DeclareMathOperator{\Rep}{Rep}

\DeclareMathOperator{\Isoc}{Isoc}
\DeclareMathOperator{\SL}{SL}
\DeclareMathOperator{\Irr}{Irr}
\DeclareMathOperator{\iso}{iso}
\DeclareMathOperator{\der}{der}
\newcommand{\wh}{\widehat}
\newcommand{\D}{\mathbb{D}}
\newcommand{\co}[1]{\prescript{#1}{}}
\newcommand{\Gm}{\mathbb{G}_m}

\title{A Tannakian description of the local Kaletha gerbe}
\date{January 2026}

\author{Alexander Bertoloni Meli}
\address{Department of Mathematics and Statistics, Boston University, 665 Commonwealth Ave, Boston, MA
02215, USA.}
\author{Peter Dillery}
  \address{Mathematics Institute, University of Bonn, Endenicher Allee 60, 53115 Bonn, Germany}

\begin{document}
\maketitle

\begin{abstract}
    We construct, for a $p$-adic field $F$, an explicit semisimple Tannakian category $\RigIsoc_{F}$ whose category of fiber functors recovers Kaletha's Galois gerbe $\mc{E}_{\Kal}$. We then classify and write down the simple objects in $\RigIsoc_{F}$, all of which come from elliptic twisted Levi subgroups of $\GL_{n}$.
\end{abstract}

\section{Introduction}

\subsection{Motivation and overview}
The refined local Langlands correspondence for a $p$-adic field $F$ posits an injection from irreducible smooth representations of $G(F)$, for $G$ reductive and quasi-split, to enhanced $L$-parameters $(\phi, \rho)$ where $\phi: \WD_{F} \rightarrow \co{L}{G}$ (here $\WD_{F} = W_{F} \times \mathrm{SL}_{2}(\C)$ is the $\SL_2$ form of the Weil--Deligne group) and $\rho$ is an algebraic irreducible representation of the centralizer group $S_{\phi}$ of the image of $\phi$. The image of the refined correspondence map should be contained in the set of pairs $(\phi, \rho)$ such that $\rho \in \Irr(\pi_0(S_{\phi}/Z(\wh{G})^{\Gamma_F}))$. Vogan (\cite{VoganLLC}) proposed that in order to upgrade the correspondence to contain $\Irr(\pi_0(S_{\phi})) \subset \Irr(S_{\phi})$ in its image, one should enlarge the domain by adding in irreducible representations of inner forms of $G$. In \cite{BMO}, Oi and the first named author showed that if a refined local Langlands correspondence exists for $G$ and its Levi subgroups, one can upgrade it to a bijection with the entire target $\Irr(S_{\phi})$ if one further allows irreducible representations of Levi subgroups of $G$ on the automorphic side.  The question of exactly how to parametrize inner forms of $G$ (and possibly its Levi subgroups) in this context is a subtle one because the cohomology group $H^1(F, G_{\ad})$ is not sufficiently rigid, in the sense that translating by $G_{\ad}$-coboundaries identifies non-isomorphic representations of $G(F)$. However, there are many options for sufficiently rigid cohomology sets including the `pure inner forms' given by $H^1(F, G)$, the `extended pure inner forms' classified by the Kottwitz set $B(G)_{\bas}$, and finally `rigid inner forms' defined by Kaletha in \cite{Kalannals} and corresponding to basic classes $H_{\bas}^1(\mc{E}_{\Kal}, G)$ of the rigid gerbe denoted $\mc{E}_{\Kal}$ in this paper. This last set has the advantage that it surjects onto $H^1(F, G_{\ad})$, whereas the others do not in general. For a nice discussion of this topic, we recommend the survey article \cite{KalSurvey}.

One advantage that the Kottwitz set $B(G)_{\bas}$ (or more generally $B(G)$) has is that there is an explicit Tannakian category, namely that of $F$-isocrystals, $\Isoc_F$, whose objects are in bijection with $\coprod_n B(\GL_n)$ for $n \in \N$. The category of $F$-isocrystals appears naturally in Dieudonn\'{e} theory and hence has connections to local Shimura varieties. By the Tannakian formalism, we can thus view the isocrystal gerbe $\mc{E}_{\text{iso}}$ (cf. \cite{LR} where this is called the Dieudonn\'{e} gerbe), whose algebraic cohomology is $B(G)$, as the Galois gerbe obtained from the gerbe of fiber functors of $\Isoc_{F}$ and, in the other direction, identify the tensor category $\Rep(\mc{E}_{\text{iso}})$ with $\Isoc_{F}$ (see \cite[\S 7.1]{Taibi25} for a summary of this formalism). 

The goal of the present work is to construct a Tannakian category of ``Rigid Isocrystals'' $\RigIsoc_{F}$ which can be canonically identified with $\Rep(\mc{E}_{\Kal})$ and whose objects are classified by $\coprod_n H^1(\mc{E}_{\Kal}, \GL_n)$ for $n \in \N$. Considered as a Galois gerbe, $\mc{E}_{\Kal}$ is an extension of $\Gamma_{\ov{F}/F}$ by $u(\ov{F})$, where $u$ is the pro-multiplicative group $\varprojlim_{E/F, n} \Res_{E/F} (\mu_n)$, where the limit is takes over all finite Galois extensions $E/F$ and all natural numbers $n$. We alert the reader that this group was originally defined in \cite{Kalannals} with $u'$ in place of $u$, where
\begin{equation*}
u' = \varprojlim_{E/F,n} \frac{\Res_{E/F} (\mu_n)}{\mu_{n}},
\end{equation*}
but one can use $u$ as well because the natural map $u \to u'$ is an isomorphism (see Remark \ref{rem: Kalbbandvariant}). Kaletha works with the $u'$-presentation of this group because it gives the correct object over $\R$ as well. Since we only consider $p$-adic fields, it suffices to consider the $u$-presentation.

The point of our article is to explain that since the gerbe $\mc{E}_{\Kal}$ can be constructed from the isocrystal gerbe with band $\Gm$ by a relatively simple sequence of operations, the same is true at the level of Tannakian categories. Let us describe the operations we need to consider. Beginning with the isocrystal gerbe with band $\Gm$, we pass to the subgroup $\mu_n$. At the level of Tannakian categories, this is a quotient operation and is realized by passing to a certain module category. After this, we take products to get a gerbe with band $\prod\limits_{\Gamma_{E/F}} \mu_{n,E}$ and then descend this gerbe to produce one with band $\Res_{E/F} (\mu_n)$. Finally, we take limits over $E, n$. By understanding just these four operations, products, modules, limits, and descent, we can produce a Tannakian category for the gerbe $\mc{E}_{\Kal}$.

\begin{thm}
    There is a Tannakian category $\RigIsoc_{F}$ whose associated gerbe of fiber functors represents $[\mc{E}_{\Kal}]$ \footnote{Of course, one can always take the Tannakian category $\Rep(\mc{E}_{\Kal})$ and recover $\mc{E}_{\Kal}$ as the associated gerbe, but the point of this result is to give a more explicit Tannakian category.} and is described as follows. Its objects consists of tuples $(V, \tilde{\Phi}, n, E)$ where $n \in \N$, and $E/F$ is a finite Galois extension. The other components of this tuple are given as follows:
    \begin{itemize}
    \item{$(V,\mc{D})$ is a vector space over $L_{E}$ ,the completion of the maximal unramified extension of $E$, equipped a $\prod\limits_{\Gamma_{E/F}} \frac{1}{n}\Z$-grading $\mc{D}$ and an action of the ring $R_{E/F} := L_{E}[\prod \limits_{\Gamma_{E/F}} \Z]$ which is additive on the gradings and makes $V$ a finitely-generated $R_{E/F}$-module;}
    \item{$\tilde{\Phi}$ is a semilinear action of $W_{L_E/F}$ on $V$ that acts in the natural way on the grading $\mc{D}$ and is compatible with the action of $W_{L_{E}/F}$ on $R_{E/F}$ (cf. \S \ref{sec:Descent} for the description of this action) ;}
    \item{
    There is an $L_E$-basis of the $(a_{\sigma})_{\sigma}$-graded piece of $V$ such that an $n$th power $\mf{f}^n_E$ of a choice $\mf{f}_{E}$ of the Frobenius of $E$ in $W_{L_{E}/F}$  acts by $\prod\limits_{\sigma \in \Gamma_{E/F}} \sigma(\varpi_E)^{a_{\sigma}}$, for $\varpi_E$ a fixed choice of anti-uniformizer.}
    \end{itemize}
    Morphisms betweem two such objects in $\RigIsoc_{F}$ are $W_{L_{E}/F}$-equivariant graded morphisms of $R_{E/F}$-modules.
\end{thm}

In the final section, we discuss the structure of the Tannakian category $\RigIsoc_{F}$, which is semisimple, and give a number of examples of its simple objects. We point out that the situation differs fairly dramatically from that of $\Isoc_F$ concerning the role of basic objects---in $\Isoc_{F}$ (resp. $\RigIsoc_{F}$) these are the objects whose graded pieces are concentrated in one $\Q$- (resp. $\prod_{\Gamma_{E/F}} \Q/\Z$-) degree and are important because they give inner forms of $\GL_{n}$. For $\Isoc_{F}$, the simple objects are all basic and are classified according  to the Dieudonn\'{e}--Manin classification by a rational slope $r/s \in \Q$ written in reduced form. A basic isocrystal is simple if $s$ equals the dimension of the underlying $L$-vector space. 

In our situation, the simple objects of $\Rep(\mc{E}_{\Kal})$ need not be basic, but instead come from certain basic classes of $H^1(\mc{E}_{\Kal}, M)$ for $M$ an elliptic twisted Levi subgroup of $\GL_n$. Recall that twisted Levi subgroups are $F$-rational subgroups that become Levi subgroups upon base-change and are precisely the centralizers of $F$-rational tori. A twisted Levi subgroup is \emph{elliptic} if the split rank of its center equals that of $G$, and the elliptic twisted Levi subgroups of $\GL_n$ are all isomorphic to $\Res_{E/F}\GL_s$ with $s[E \colon F] =n$. The set of basic classes of such a twisted Levi are in bijection with tuples $(a_i) \in (\Q/\Z)^{\Hom_{F}(E,\ov{F})}$ such that $\sum\limits_i a_i \in \frac{1}{s} \Z / \Z$. In this paper we explain how to explicitly construct the simple object in $\RigIsoc_{F}$ corresponding to each such $E$ and $(a_{i})$. More precisely, we prove:

\begin{thm} 
Every isomorphism class of simple objects in $\RigIsoc_{F}$ may be represented by a pair $(E, (a_{i}))$ where $E/F$ is a finite separable field extension and the tuple $(a_{i})$ is an element of $(\Q/\Z)^{\Hom_{F}(E,\ov{F})}$ which is not inflated from a tuple in $(\Q/\Z)^{\Hom_{F}(K,\ov{F})}$ for any proper subextension $F/K/E$. For a fixed collection of finite extensions $\{E\}$ which have the same Galois closure $\widetilde{E}$ and are an orbit under $\Gamma_{\widetilde{E}/F}$, the group $\Gamma_{\widetilde{E}/F}$ acts on all such pairs involving $E' \in \{E\}$, and two arbitrary pairs give the same isomorphism class if and only if the extensions $E_{1}, E_{2}$ are related in this manner and the pairs lie in the same $\Gamma_{\widetilde{E}/F}$-orbit.
 \end{thm}

In the case of $\mc{E}_{\iso}$, the explicit linear-algebraic description of $\Rep(\mc{E}_{\text{iso}})$ has many uses, such as in the definition of crystalline cohomology, Dieudonn\'{e} theory for $p$-divisible groups, and in the definition of the stack $\Isoc_G$ in Zhu's (\cite{Zhu25}) geometrization of the local Langlands correspondence. It is reasonable to ask analogous questions of $\RigIsoc_F$; namely if there is a natural cohomology theory valued in this Tannakian category and if this category leads to a more refined geometrization of the local Langlands correspondence. We do not take up these questions here, but we do relate $\RigIsoc_F$ to Fargues' category $\Isoc_{F}^e$ of extended isocrystals \cite{Fargues} and remark that steps towards a refined geomerization are taken in that paper.

\subsection{Preliminary notation}
We fix the following notation
\begin{itemize}
    \item We fix $F$ a $p$-adic local field.
    \item Let $L$ be the completion of the maximal unramified extension of $F$ and fix an algebraic closure $\ov{L}$. 
    \item Let $\ov{F}$ be the algebraic closure of $F$ in $\ov{L}$.
    \item We will consider finite Galois extensions $E/F$ and denote by $L_E \subset \ov{L}$ the completion of the maximal unramified extension of $E$.
    \item Denote by $W_{L_E/E}$ the subgroup of $\Gamma_{L_E/E}$ generated by integral powers of a fixed choice of Frobenius morphism which we denote by $\mf{f}_E$.
    \item For each $E$, fix $\varpi_E$ an anti-uniformizer of $E$ such that if $K/E/F$ and $K'/E$ is the maximal unramified subextension, we have $N_{K/K'}(\varpi_{K}) = \varpi_{E}$ (this exists by \cite[Lemme 5.5]{Fargues}).
    \item Let $E_{s}$ denote the unramified extension of $E$ of degree $s \in \mathbb{N}$.
\end{itemize}

\subsection{Acknowledgments}
The authors thank Tasho Kaletha and David Schwein for many helpful conversations. This work was started while both authors were supported by NSF grant DMS-1840234 at the University of Michigan. The first named author is partially supported by NSF grant DMS-2502131.

\tableofcontents

\section{Construction of Tannakian Categories}
All Tannakian categories constructed in this section are built from the category of isocrystals, whose objects are pairs $(V,\Phi)$, where $V$ is a finite-dimensional $L$-vector space and $\Phi$ is a $\mf{f}_{F}$-semilinear automorphism and whose morphisms are $\Phi$-equivariant linear maps. For a fixed $s \in \N$, one can also consider the full subcategory of all $(V,\Phi)$ which have a basis on which $\mf{f}_{F}^{s}$ acts on each basis element by a (possibly different) power of $\varpi_{F}$.
\subsection{Products of gerbes}\label{sec:prod}
\begin{defn}
Fix $s \in \N$. Then we define the category $P_{E/F, s}$ as follows. 
\begin{itemize}
    \item Objects are triples $(V, \Phi, \mc{D})$, where $V$ is a finite-dimensional $L_{E}$-vector space, $\Phi \colon V \to V$ is an $\mf{f}_{E}$-semilinear automorphism, and $\mc{D}$ is a $\prod\limits_{\Gamma_{E/F}} \frac{1}{s}\Z$-grading on $V$ such that the subspace $V_{(b_{\sigma}/s)_{\sigma}}$ has an $L_{E}$-basis on which $\Phi^{s}$ acts by the scalar $\prod\limits_{\gamma \in \Gamma_{E/F}} \gamma(\varpi_{E})^{b_{\gamma}}$.
    \item  Morphisms in $P_{E/F,s}$ will be $\Phi$-equivariant morphisms of $L_{E}$-spaces which preserve the gradings.
\end{itemize}
\end{defn}
\begin{lem}
 The category $P_{E/F, s}$ can be equipped with the structure of a Tannakian category over $E$ in a natural way. The forgetful functor $\omega_{P_{E/F, s}}: P_{E/F, s} \to \Vect_{L_E}$ is a fiber functor defined over $L_E$.
\end{lem}
\begin{proof}
 There is a natural functor $\otimes \colon P_{E/F,s} \times P_{E/F,s} \to P_{E/F,s}$ which sends $(V, \Phi, \mc{D}) \times (V', \Phi', \mc{D}')$ to the triple $(V \otimes_{L_{E}} V', \Phi \otimes \Phi', \mc{D} \otimes \mc{D}')$, where the first two terms are clear and the grading $\mc{D} \otimes \mc{D}'$ is given by 
$$(V \otimes V')_{(b_{\sigma}/s)_{\sigma}} = \bigoplus_{(c_{\sigma}/s)_{\sigma} \in \prod\limits_{\Gamma_{E/F}} \frac{1}{s}\Z} V_{(c_{\sigma}/s)_{\sigma}} \otimes_{L_{E}} V'_{(\frac{b_{\sigma} - c_{\sigma}}{s})_{\sigma}}.$$

The above construction gives a well-defined tensor functor on $P_{E/F,s}$. Denote by $\textbf{1}_{P_{E/F,s}}$ the object given by the vector space $L_{E}$ with the usual $\mf{f}_{E}$-action concentrated in the $(0)_{\sigma}$-graded piece. Then $P_{E/F,s}$  is a rigid abelian tensor category with identity object $\textbf{1}_{P_{E/F,s}}$. Moreover, the endomorphism algebra of $\textbf{1}_{P_{E/F,s}}$ is identified with $E$. Finally, we note that the  functor $\omega_{P_{E/F,s}} \colon P_{E/F,s} \to \text{Vect}_{L_{E}}$ is an $E$-linear exact tensor functor, making $P_{E/F,s}$ a Tannakian category over $E$ which is neutralized by the extension $L_{E}$. 
\end{proof}

If $\Gamma$ is a profinite group acting continuously on an abelian group $A$ and we fix $c \in Z^{2}(\Gamma,A)$, there is a standard way to construct a group extension of $\Gamma$ by $A$ which we use throughout this paper. Namely, one defines $\mc{E}$ as a set by $A \times \Gamma$, with group multiplication given by $(a,g) \cdot (b,h) = (a \cdot \prescript{g}{}b \cdot c(g,h), gh)$. Observe that there is a natural set-theoretic section $\Gamma \to \mc{E}$ sending $g$ to $(1,g)$.

We now identify the gerbe corresponding to the category $P_{E/F, s}$. For each $\sigma \in \Gamma_{E/F}$, the element $\sigma(\varpi_{E})$ determines a cocycle of $\Gamma_{E_{s}/E}$ representing the anti-canonical class in $H^{2}(\Gamma_{E_{s}/E}, E_{s}^{\times})$, given explicitly by $c_{\sigma}(\mf{f}_{E}^{i}, \mf{f}_{E}^{j}) = 1$ if $i+j < s$, and $\sigma(\varpi_{E})$ otherwise. We will consider a group extension corresponding to the 2-cocycle $(c_{\sigma})_{\sigma}$ of $\Gamma_{E_{s}/E}$ valued in $\prod_{\Gamma_{E/F}} \mathbb{G}_{m}(E_{s})$, given by 
$$1 \to  \prod_{\Gamma_{E/F}} \mathbb{G}_{m}(E_{s}) \to  \widetilde{\prod}_{\Gamma_{E/F}} \mc{E}_{c_{\sigma}} \to \Gamma_{E_{s}/E} \to 1,$$ where by ``$\widetilde{\prod}$'' we mean that the product is fibered with respect to the projection of each $\mc{E}_{c_{\sigma}}$ to $\Gamma_{E_{s}/E}$.  Denote this Galois gerbe by $\mc{E}_{(c_{\sigma})_{\sigma}}$.

We now make the following claim:
\begin{claim}
The category $P_{E/F, s}$ is isomorphic as Tannakian categories over $E$ to $\text{Rep}_{E_{s}}(\mc{E}_{(c_{\sigma})_{\sigma}})$.
\end{claim}
Before proving the claim, we first investigate the latter Tannakian category. Let $\rho \colon \mc{E}_{(c_{\sigma})_{\sigma}} \to \text{Aut}(V)$ be a representation of $\mc{E}_{(c_{\sigma})_{\sigma}}$, for a finite-dimensional $E_{s}$-vector space $V$ . The Galois gerbe $\mc{E}_{(c_{\sigma})_{\sigma}}$ comes, by construction, equipped with a section $\Gamma_{E_{s}/F} \to \mc{E}_{(c_{\sigma})_{\sigma}}$, which allows us to identify it with pairs $((a_{\sigma}), x)$ for $a_{\sigma} \in E_{s}^{\times}$ for each $\sigma \in \Gamma_{E/F}$ and $x \in \Gamma_{E_{s}/E}$, where the group operation is $$((a_{\sigma})_{\sigma}, x)((b_{\sigma})_{\sigma},y) = ((a_{\sigma}x(b_{\sigma})c_{\sigma}(x,y))_{\sigma}, xy).$$

Our fixed representation $\rho$ gives an $\mf{f}_{E}$-semilinear automorphism $\Phi := \rho((1_{\sigma}), \mf{f}_{E})$ of $V$, which is such that $\Phi^{s}$ acts by $\rho((\sigma(\varpi_{E})), 1)$; note also that $\rho \colon \prod_{\Gamma_{E/F}} \mathbb{G}_{m}(E_{s}) \to \mathrm{GL}_{V}$ must be an algebraic homomorphism. Via simultaneous diagonalization, we may decompose $V$ into a direct sum $V = \bigoplus\limits_{\prod_{\Gamma_{E/F}} \mathbb{Z}} V_{(h_{\sigma})_{\sigma}}$, where $(\mathbb{G}_{m})_{\tau}$ acts on $V_{(h_{\sigma})_{\sigma}}$ via $x \cdot v = x^{h_{\tau}}v$ for each $x \in (\mathbb{G}_{m})_{\tau}(E_{s})$. Note that each subspace $V_{(h_{\sigma})}$ is $\Phi$-stable, and also that $\Phi^{s}$ must act on $V_{(h_{\sigma})}$ by the scalar $\prod_{\Gamma_{E/F}} \sigma(\varpi_{E})^{h_{\sigma}}$. A morphism of $\mc{E}_{(c_{\sigma})}$-representations from $V$ to $W$ sends $V_{(h_{\sigma})}$ to $W_{(h_{\sigma})}$ for all tuples $(h_{\sigma})$ and is $\Phi$-equivariant.

We now prove the claim.
\begin{proof}
We first define a functor from $\text{Rep}_{E_{s}}(\mc{E}_{(c_{\sigma})})$ to $P_{E/F,s}$; for a representation $(V, \rho)$, we send it to $(V \otimes_{E_{s}} L_{E}, \Phi \otimes \text{id}_{L_{E}}, \mc{D})$, where $\mc{D}$ is the grading induced by tensoring the $\prod_{\Gamma_{E/F}} \frac{1}{s} \Z$-grading $V_{(h_{\sigma}/s)} := V_{(h_{\sigma})}$ with $L_{E}$ over $E_{s}$. A morphism from $(V, \rho)$ to $(W, \rho')$ will go to the morphism given by applying the functor $- \otimes \text{id}_{L_{E}}$; this construction evidently defines a morphism of Tannakian categories over $E$. There is also an inverse functor, given by sending the tuple $(V, \Phi, \mc{D})$ to the representation given as a vector space $V'$ by the $E_{s}$-span of any choice of $L_{E}$-bases of $V_{(h_{\sigma}/s)}$ (for all tuples) such that $\Phi^{s}$ acts by $\prod_{\Gamma_{E/F}} \sigma(\varpi_{E})^{h_{\sigma}}$. Note that this vector space does not depend on the choice of basis: Indeed, given any $v$ on which $\Phi^{s}$ acts by $\lambda := \prod \sigma(\varpi_{E})^{h_{\sigma}}$ and choice of basis $(e_{i})$ for the corresponding graded piece, writing $v = \sum a_{i}e_{i}$ for $a_{i} \in L_{E}$ gives that $$\sum (a_{i} \lambda)e_{i} = \sum (\mf{f}_{E}^{s}(a_{i})\lambda) e_{i},$$ which means that $\mf{f}_{E}^{s}(a_{i}) = a_{i}$ for all $i$, so that $a_{i} \in E_{s}$. 

The map $\rho$ is defined by sending $((1_{\sigma}), \rho)$ to $\Phi$ and $(\mathbb{G}_{m})_{\tau}$ to the linear operator on $V'$ determined by $x \cdot v = x^{h_{\sigma}}v$ for any $v$ in the $E_{s}$-span of a basis of the graded component $V_{(h_{\sigma}/s)}$. This is a well-defined representation, since for any $(a_{\sigma})\in \prod_{\Gamma_{E/F}} \mathbb{G}_{m}(E_{s})$, we have that the maps $\Phi \circ (a_{\sigma}) \circ \Phi^{-1}$ and $\mf{f}_{E} (a_{\sigma}) \mf{f}_{E}^{-1}$ are $L_{E}$-linear and act on a vector element $v$ in $V'_{(h_{\sigma}/s)}$ as follows: $$\mf{f}_{E} (a_{\sigma}) \mf{f}_{E}^{-1}v = ((\mf{f}_{E}(a_{\sigma})),1) \cdot v = (\prod_{\Gamma_{E/F}} \mf{f}_{E}(a_{\sigma})^{h_{\sigma}})v,$$ and we also have (using the $\Phi$-stability of $V'_{(h_{\sigma})}$) that $$(\Phi \circ (a_{\sigma}) \circ \Phi^{-1})v = \Phi( (a_{\sigma}) \cdot (\Phi^{-1}(v))) = \Phi(\prod a_{\sigma}^{h_{\sigma}}\Phi^{-1}(v)) = (\prod \mf{f}_{E}(a_{\sigma})^{h_{\sigma}})v,$$ as desired. These define mutually-inverse functors, giving the desired equivalence.
\end{proof}

\begin{cor}\label{productgerbe} The gerbe over $\text{Spec}(E)_{\text{fpqc}}$ corresponding to the Tannakian category $P_{E/F,s}$ is isomorphic to the product of the gerbes (as fibered categories over $\text{Spec}(E)_{\text{fpqc}}$) $\prod_{\sigma \in \Gamma_{E/F}} \mc{E}_{c_{\sigma}}$. In particular, it is banded by the $E$-group scheme $\prod_{\Gamma_{E/F}} \mathbb{G}_{m}$ and corresponds to the anti-canonical class in $H^{2}(\Gamma_{L_{E}/E}, \prod_{\Gamma_{E/F}} L_{E}^{\times})$ corresponding to the element $(-1/s)_{\sigma} \in \prod_{\Gamma_{E/F}} \mathbb{Q}/\Z$.
\end{cor}

What the above construction does is give an isomorphism of $\prod_{\Gamma_{E/F}}\mathbb{G}_{m}$-gerbes from $\mc{E}_{P_{E/F,s}}$ to $\mc{E}_{(c_{\sigma})_{\sigma}}$; note that since $H^{1}(\Gamma_{L_{E}/E}, \prod_{\Gamma_{E/F}} L_{E}^{\times})$ vanishes, this isomorphism is unique up to conjugation by an element of $\prod_{\Gamma_{E/F}} L_{E}^{\times}$, which we will see later does not affect cohomology.

\subsection{Limits of gerbes}
For $s \mid s'$, we have a fully faithful embedding of categories from $\iota_{s,s'} \colon P_{E/F,s}$ to $P_{E/F,s'}$, which sends the triple $(V, \Phi, \mathcal{D})$ to the triple $(V, \Phi, \mathcal{D}_{s'/s})$, where $\mathcal{D}_{s'/s}$ is the ${\prod_{\Gamma_{E/F}}\frac{1}{s'}\mathbb{Z}}$-grading of $V$ given by $V_{(b_{\sigma}/s')} = V_{([b_{\sigma}/(s'/s)]/s)}$ if $s'/s$ divides $b_{\sigma}$ for all $\sigma$, and is zero otherwise. 

We define the category $P_{E/F}$ to be the direct limit $\varinjlim_{s} P_{E/F,s}$ over all positive integers $s$ (direct limit of categories in the obvious sense). Explicitly, a morphism between the objects $(V, \Phi, \mathcal{D}, s)$ and $(V', \Phi', \mathcal{D}', s')$ can be identified with a morphism between $\iota_{s, ss'}(V, \Phi, \mathcal{D}, s)$ and $\iota_{s', ss'}(V', \Phi', \mathcal{D}', s')$ in the category $P_{E/F,ss'}$. 

Note that if one looks at a subspace $V_{(b_{\sigma}c/sc)}$, viewed as an object of $P_{E/F,sc}$, then although a priori it only has a basis on which $\mf{f}_{E}^{sc}$ acts by $\prod_{\sigma \in \Gamma_{E/F}} \sigma(\varpi_{E})^{b_{\sigma}c}$, one can say even more: It has an $L_{E}$-basis on which $\mf{f}_{E}^{s}$ acts by $\prod_{\sigma \in \Gamma_{E/F}} \sigma(\varpi_{E})^{ b_{\sigma}}$, as follows from the general theory of $\Phi$-$L_{E}$ spaces (cf. \cite[p.11]{Kottwitzisoci}). This shows that the embedding $\iota_{s,s'}(P_{E/F,s})$ may be described as all tuples whose only non-zero graded components have indices which may be written as $(\frac{b_{\sigma}(s'/s)}{s(s'/s)})_{\sigma}$. The identification of the gerbe corresponding to $P_{E/F,s}$ with $\mc{E}_{(c_{\sigma})_{\sigma}}$ is compatible for varying $s$, provided that we choose the sections compatibly (which is always possible), and hence we see that the Tannakian category $P_{E/F}$ is banded by $\prod_{\Gamma_{E/F}} \mathbb{D}(L_{E})$ (where $\D$ is the pro-torus with character group $\Q$), and corresponds to the anti-canonical class in $H^{2}(\Gamma_{L_{E}/E}, \prod_{\Gamma_{E/F}} \mathbb{D}(L_{E}))$ given by the inverse limit of the finite anti-canonical classes (these transition maps are given for $s \mid s'$ by the multiplication-by-$s'/s$ map $\frac{1}{s'}\Z/\Z \to \frac{1}{s}\Z/\Z$). 

For $s \mid s'$, it is worth investigating the transition maps on bands corresponding to the map of extensions $p_{s',s} \colon \mc{E}_{P_{E/F,s'}} \to \mc{E}_{P_{E/F,s}}$. First, we need to say what this map of extensions is: Note that, by base-changing to the gerbe $\mc{E}_{(c_{\sigma})_{\sigma}}$ to $\ov{F}$, we get the extension 
\begin{equation*} 1 \to \prod_{\Gamma_{E/F}} \mathbb{G}_{m}(\ov{F}) \to \mc{E}_{(c'_{\sigma})_{\sigma}} \to \Gamma_{\ov{F}/E} \to 1,
\end{equation*}
where each $c_{\sigma}'$ denotes the inflated anti-canonical class $\Gamma_{\ov{F}/E}^{2} \to \Gamma_{E_{s}/E}^{2} \xrightarrow{c_{\sigma,s}} \prod_{\Gamma_{\ov{F}/E}} E_{s}^{\times}$---we can do the identical thing for $s'$. We then get a morphism of group extensions 
\[
\begin{tikzcd}
1 \arrow{r} & \prod_{\Gamma_{E/F}} \mathbb{G}_{m}(\ov{F}) \arrow{r} \arrow["s'/s"]{d} & \mc{E}_{(c_{\sigma,s'}')_{\sigma}} \arrow{d} \arrow{r} & \Gamma_{\ov{F}/E} \arrow[equals]{d} \to 1 \\
1 \arrow{r} & \prod_{\Gamma_{E/F}} \mathbb{G}_{m}(\ov{F}) \arrow{r} & \mc{E}_{(c_{\sigma,s}')_{\sigma}} \arrow{r} & \Gamma_{\ov{F}/E} \to 1,
\end{tikzcd}
\]
and we get a commutative square of Tannakian categories 
\[
\begin{tikzcd}
P_{E/F,s',\ov{L}} \arrow{r} & \text{Rep}_{\ov{L}}(\mc{E}_{(c_{\sigma,s'}')_{\sigma}}) \\
P_{E/F,s,\ov{L}} \arrow{r} \arrow{u} & \text{Rep}_{\ov{L}}(\mc{E}_{(c_{\sigma,s}')_{\sigma}}), \arrow{u}
\end{tikzcd}
\]
where the left-hand map is the transition map described above, the right-hand map is the one induced by the above morphism of extensions, and the horizontal ones are induced by the equivalence of Tannakian categories described above. 

By $P_{E/F,s,\ov{L}}$ we mean the category of $\prod_{\Gamma_{E/F}} \frac{1}{s}\mathbb{Z}$-graded finite-dimensional $\ov{L}$-vector spaces $V$ with a $W_{E}$-semilinear action, such that each $V_{(a_{\sigma}/s)}$ is $W_{E}$-stable and $(V_{(a_{\sigma}/s)})^{\Gamma_{\ov{L}/L_{E}}}$ has an $L_{E}$-basis (which implies that it is an $\ov{L}$-basis of $V_{(a_{\sigma}/s)}$) on which $\mf{f}_{E}^{s}$ acts by the scalar $\prod_{\Gamma_{E/F}} \sigma(\varpi_{E})^{a_{\sigma}}$. By $\text{Rep}_{\ov{L}}(\mc{E}_{(c_{\sigma,s}')})$, we mean the Tannakian category of representations $(V, \rho)$ (over $\ov{L}$) of $\mc{E}_{(c_{\sigma,s}')}$. The equivalence of the previous subsection goes through for these two categories after some elementary modifications. We check that, starting with an object $(V, \Phi, \mc{D})$, going right and then up gives the representation $(V, \rho)$, where $\rho$ sends $\mf{f}_{E}$ to $\Phi$, and the corresponding $E$-morphism $\prod \mathbb{G}_{m} \to \mathrm{GL}_{V}$ is determined by acting on the subspace $V_{(h_{\sigma}/s)} \subseteq V$ via $$(a_{\sigma}) \cdot v = \prod a_{\sigma}^{h_{\sigma}(s'/s)}v.$$

On the other hand, going up and then right gives the representation $(V, \rho')$, where $\rho'$ sends $\mf{f}_{E}$ to $\Phi$ with $E$-morphism determined by acting on $V_{(h'_{\sigma}/s')}$ by $$(a_{\sigma}) \cdot v= \prod a_{\sigma}^{h'_{\sigma}}v.$$ Now the result follows from the fact that $h'_{\sigma} = h_{\sigma}(s'/s)$ for all $\sigma$, by construction of our transition functor. 

\begin{cor} The transition morphism of Tannakian categories $P_{E/F,s} \to P_{E/F,s'}$ gives the $s'/s$-power morphism on bands (via the identifications of the above Corollary \ref{productgerbe}, using compatible sections).
\end{cor}

\subsection{Modules}\label{sec:mod}
Now, let $R_{E/F}$ denote the $\prod_{\Gamma_{E/F}}\Z$-graded $L_{E}$-algebra $$\bigoplus_{(b_{\tau})_{\tau} \in \prod_{\Gamma_{E/F}}\Z} (L_{E})_{(b_{\tau})_{\tau}},$$ with addition defined component-wise and multiplication determined by $\ell_{(b_{\tau})_{\tau}} \cdot \ell'_{(c_{\tau})_{\tau}} = (\ell \cdot \ell')_{(b_{\tau} + c_{\tau})_{\tau}}$ for all $\ell, \ell' \in L_{E}$. This ring will also be equipped with an $\mf{f}_{E}$-semilinear action, denoted by $\Phi_{R_{E/F}}$, determined by the semilinearity condition and the identity $\Phi_{R_{E/F}}(1_{(b_{\sigma})}) = \prod_{\gamma \in \Gamma_{E/F}} \gamma(\varpi_{E})^{b_{\gamma}}1_{(b_{\sigma})}$.
\begin{defn}
For a positive integer $s$, we define the category $RP_{E/F,s}$ to have objects given by tuples $(V, \Phi, \mc{D})$, consisting of
\begin{itemize}
    \item a (not necessarily finite-dimensional) $L_{E}$-vector space $V$  equipped with an $R_{E/F}$-action such that it is finitely generated as an $R_{E/F}$-module,
    \item an $\mf{f}_{E}$-semilinear isomorphism $\Phi$, 
    \item and a $\prod_{\Gamma_{E/F}} \frac{1}{s}\Z$-grading $\mc{D}$ such that $V_{(a_{\sigma}/s)}$ has an $L_{E}$-basis on which $\Phi^{s}$ acts by the scalar $\prod_{\sigma \in \Gamma_{E/F}} \sigma(\varpi_{E})^{a_{\sigma}}$.
\end{itemize}

In addition, we require that the $R_{E/F}$-action is compatible with the $\Phi$-action, in the sense that, for all $v \in V$ and $r \in R$, we have $\Phi(r \cdot v) = \Phi_{R_{E/F}}(r) \cdot \Phi(v)$. Note that this last condition implies that $$1_{(b_{\sigma})} \cdot V_{(a_{\sigma}/s)} \subseteq V_{(\frac{a_{\sigma} + sb_{\sigma}}{s})}.$$

Morphisms in $RP_{E/F,s}$ will be $\Phi$-equivariant morphisms of $R_{E/F}$-modules that preserve the gradings. Note that $RP_{E/F,s}$ is a rigid abelian tensor category over $E$ with identity object $R_{E/F}$ (with the canonical $R_{E/F}$-module structure, and $\Phi = \Phi_{R_{E/F}}$).
\end{defn}

We now define a fiber functor for this category. First, note that there is a canonical surjective augmentation map $R_{E/F} \to L_{E}$ defined by sending each $1_{(b_{\sigma})}$ to $1$, whose kernel we will denote by $I$. Note for an object $(V, \Phi, \mc{D})$ of $RP_{E/F,s}$, that $V$ is a finitely generated and free, hence flat, $R_{E/F}$-module. Hence, we have an equality of $L_{E}$-vector spaces $V \otimes_{R} R_{E/F}/I = V/IV$. We define a functor $\omega_{RP_{E/F, s}}: RP_{E/F, s} \to \Vect_{L_E}$ by
\begin{equation*}
    (V, \Phi, \mc{D}) \mapsto V/IV.
\end{equation*}

Note that $IV$ is exactly the kernel of the natural $L_{E}$-linear surjection $V \to \bigoplus_{(0) \leq (a_{\sigma}/s) < (1)} V_{(a_{\sigma}/s)}$ defined on each $V_{([a_{\sigma}+sb_{\sigma}]/s)}$ by multiplication by $1_{(-b_{\sigma})}$. This defines an $L_{E}$-linear tensor functor. 

\begin{lem}\label{exactlem1} The above functor is a fiber functor; that is, it is faithful and exact.
\end{lem}

\begin{proof} We first prove faithfulness. Let $f$ and $g$ be two morphisms from $V$ to $W$ (we omit the other structure from our notation for brevity), and suppose that their image under this functor coincides, and let $V_{(b_{\sigma}/s)}$ be a fixed graded component of $V$. Since these categories are additive, we may assume that $g$ equals zero. For any $v \in  V_{(b_{\sigma}/s)}$, we know that $f(v) \in W_{(b_{\sigma}/s)}$, but also that $f(v) \in IW$, and thus conclude by observing that $IW \cap W_{(b_{\sigma}/s)}= 0$ for all tuples $(b_{\sigma}/s)$.

We now prove exactness. By construction, this functor is evidently right-exact, so we only need to prove that it preserves injections. To this end, suppose that $v \in V$ is such that $f(v) \in IW$ for injective $f$. Writing $$v = \sum_{(a_{\sigma}/s)} v_{(a_{\sigma}/s)}$$ in terms of graded components, we note that, since $f$ is $R_{E/F}$-linear, for all $(b_{\sigma})$, we have $$f(v_{([a_{\sigma} + sb_{\sigma}]/s)}) = 1_{(b_{\sigma})} \cdot f(1_{(-b_{\sigma})} \cdot v_{([a_{\sigma} + sb_{\sigma}]/s)}).$$ It follows that $$f(v) = \sum_{(0) \leq (a_{\sigma}/s) < (1)} \sum_{(b_{\sigma})} 1_{(b_{\sigma})} \cdot f(1_{(-b_{\sigma})} \cdot v_{([a_{\sigma} + sb_{\sigma}]/s)}).$$ 

For this expression to lie in $IW$, it is necessary that, for a fixed $(0) \leq (a_{\sigma}/s) < (1)$, we have $$\sum_{(b_{\sigma}) \neq (0)} f(1_{(-b_{\sigma})} \cdot v_{([a_{\sigma} + sb_{\sigma}]/s)}) = -f(v_{(a_{\sigma}/s)})$$ inside the space $W_{(a_{\sigma}/s)}$. Since $f$ is injective, this means that $$\sum_{(b_{\sigma}) \neq (0)} 1_{(-b_{\sigma})} \cdot v_{([a_{\sigma} + sb_{\sigma}]/s)} = -v_{(a_{\sigma}/s)}$$ in $V_{(a_{\sigma}/s)}$, which implies that $v \in IV$.
\end{proof}

It follows from the above Lemma \ref{exactlem1} that $RP_{E/F,s}$ is a Tannakian category over $E$. For $s \mid s'$, note that the embedding functor $\iota_{s,s'}$ extends naturally to a functor, also denoted by $\iota_{s,s'}$, from $RP_{E/F,s}$ to $RP_{E/F,s'}$, defined identically (with the $R_{E/F}$-module structure of $\iota_{s,s'}(X)$ being the same as the one associated to $X$). As before, the direct limit of these categories will be denoted by $RP_{E/F}.$

We now want to investigate the band of $RP_{E/F,s}$ (and thus $RP_{E/F}$). There is an ``extension of scalars" functor $q \colon P_{E/F,s} \to RP_{E/F,s}$ given by $V \mapsto R_{E/F} \otimes_{L_{E}} V$ on objects, and by $f \mapsto \text{id}_{R_{E/F}} \otimes f$ on morphisms, which is an $L_{E}$-linear tensor functor. Recall that an exact tensor functor between Tannakian categories $P$ and $Q$ is called a \textit{quotient functor} if every object of $Q$ is a subquotient of an object in the image of the functor. 

\begin{prop} The functor $q \colon P_{E/F,s} \to RP_{E/F,s}$ is a quotient functor.
\end{prop}

\begin{proof} Since $R_{E/F}$ is a flat $L_{E}$-module, the functor $q$ is evidently exact. To show the other condition for quotient functors, let $(V, \Phi, \mathcal{D})$ be an object of $RP_{E/F,s}$, and let $\{e_{i}\}_{i=1}^{n}$ be a set of generators for $V$ as an $R_{E/F}$-module, which we may assume without loss of generality is such that $e_{i} \in V_{(a_{\sigma})_{\sigma}}$ (depending on $i$) on which $\Phi^{s}$ acts in the appropriate way. We then get an object of $P_{E/F,s}$ by taking the $L_{E}$-span of $\{e_{i}\}$, with the obvious inherited grading and $\Phi$-action; denote this object by $(\tilde{V}, \tilde{\Phi}, \tilde{\mathcal{D}})$. Then $R_{E/F} \otimes_{L_{E}} \tilde{V}$ is a free $R_{E/F}$-module, with basis $\{1 \otimes e_{i}\}$, and there is a quotient map $R_{E/F} \otimes_{L_{E}} \tilde{V} \to V$ sending $1 \otimes e_{i}$ to $e_{i}$, which is a morphism in $RP_{E/F,s}$, giving the result.
\end{proof}

In the notation of \cite[\S 2]{MilneQuot}, recall that $P_{E/F,s}^{q}$ denotes all objects of $P_{E/F,s}$ whose image under $q$ is isomorphic to a direct sum of copies of the identity object in $RP_{E/F, s}$, which is a full Tannakian subcategory over $E$. Observe that an object $(V, \Phi, \mc{D})$ in $P_{E/F,s}$ lies in $P_{E/F,s}^{q}$ if and only if the only nonzero graded components $V_{(a_{\sigma}/s)}$ are such that each $a_{\sigma}/s$ is an integer. According to \cite[\S 2]{MilneQuot}, the band of the gerbe corresponding to $RP_{E/F,s}$ is canonically isomorphic to all automorphisms of the fiber functor of $P_{E/F,s}$ which restrict trivially to $P_{E/F,s}^{q}$. In order to investigate how the band acts on the restriction of the fiber functor to these objects, it is easiest to use our identification of $P_{E/F,s}$ with $\text{Rep}_{L_{E}}(\mc{E}_{(c_{\sigma})_{\sigma}})$ from the previous subsection. We will work not with $P_{E/F,s}$, but instead with $P_{E/F,s,\ov{L}}$, which is the same category as $P_{E/F,s}$ except replacing $L_{E}$-vector spaces with $\ov{L}$-vector spaces and the semilinear $\mf{f}_{E}$-action with a semilinear $W_{\ov{F}/E}$-action that acts on an $\ov{L}$-basis of the graded pieces in the same way (after restricting to $\Gamma_{L_{E}/E}$).

Let $(x_{\sigma}) \in \prod_{\Gamma_{E/F}} \mathbb{G}_{m}(\ov{F})$, which acts on the canonical $\ov{F}$-fiber functor of $\text{Rep}_{\ov{F}}(\mc{E}_{(c_{\sigma}),\ov{F}})$ by, for a fixed representation $V$, mapping the $\prod_{\Gamma_{E/F}}\Z$-graded piece $V_{(h_{\sigma})}$ to itself via sending $v$ to $(\prod x_{\sigma}^{h_{\sigma}})v$. The image of the subcategory $P_{E/F,s,\ov{F}}^{q}$ under the above equivalence is all representations whose only nonzero graded pieces have each $h_{\sigma}$ divisible by $s$. These two observations make show that the subgroup of $\prod_{\Gamma_{E/F}} \mathbb{G}_{m}(\ov{F})$ fixing the image of $P_{E/F,s,\ov{F}}^{q}$ is exactly $\prod_{\Gamma_{E/F}} \mu_{s}(\ov{F})$, which implies (since we know that the band is an $E$-subgroup scheme of the band of $\mc{E}_{(c_{\sigma})_{\sigma}}$, which is $\prod_{\Gamma_{E/F}} \mathbb{G}_{m})$) that the band of the gerbe over $E$ corresponding to the Tannakian category $RP_{E/F,s}$ may be identified with the $E$-group scheme $\prod_{\Gamma_{E/F}} \mu_{s}$.

\subsection{Descent}\label{sec:Descent}
We want to descend the Tannakian category $RP_{E/F,s}$ to a Tannakian category over $F$. The first step is to define a canonical $\Gamma_{L_E/F}$-action on the ring $R_{E/F}$. 

Note that $L_E/L$ is a totally ramified extension and that the sequence:
\begin{equation*}
    1 \to \Gamma_{L_E/L} \to W_{L_E/F} \to W_{L/F} \to 1,
\end{equation*}
is split by any lift of the Frobenius element in $W_{L/F}$. Fix such a splitting and hence a lift which we denote $\mf{f}_F$. Restriction gives an injection $\Gamma_{L_E/L} \hookrightarrow \Gamma_{E/F}$. We denote the image of this injection by $\Omega_E$.

We now define a $\Gamma_{L_E/F}$-semilinear action of $W_{L_E/F}$ on $R_{E/F}$ which we denote $\widetilde{\Phi}_{R_{E/F}}$. We define the action of $\Gamma_{L_E/L}$ such that $\gamma(\ell_{(b_{\sigma})_{\sigma}}) = \gamma(\ell)_{(b_{\ov{\gamma}^{-1}\sigma})_{\sigma}}$, where $\ov{\gamma}$ is the projection of $\gamma$ to $\Gamma_{E/F}$. We define the action of $\mf{f}_F$ such that
\begin{equation*}
\mf{f}_F(\ell_{(b_{\sigma})_{\sigma}}) = ({\mf{f}_F(\ell)\prod\limits_{\sigma' \in \Omega_E} \sigma'(\varpi_E)^{b_{\ov{\mf{f}_F}^{-1}\sigma'}}})_{(b_{\ov{\mf{f}_F}^{-1}\sigma})_{\sigma}}.
\end{equation*}

In particular, this implies that 
\begin{equation*}
\mf{f}_F^{-1}(\ell_{(b_{\sigma})_{\sigma}}) = ({\mf{f}_F^{-1}(\ell)\prod\limits_{\sigma' \in \Omega_E} \mf{f}_{F}^{-1}\sigma'(\varpi_E)^{-b_{\sigma'}}})_{(b_{\ov{\mf{f}_F}\sigma})_{\sigma}},
\end{equation*}
which can be verified easily by checking that $\mf{f}_{F}^{-1} \circ \mf{f}_{F} = \text{id}$.

We need to check that this indeed induces an action of $W_{L_E/F}$. To do so, we need to check that for $\gamma \in \Gamma_{L_E/L}$, the action of $\mf{f}_F \circ \gamma \circ \mf{f}^{-1}_F \in \Gamma_{L_E/L}$ agrees with the action induced by the composition. Since both actions are semilinear in the correct way, we need only understand the action on $1_{(b_{\sigma})_{\sigma}}$. We have 
\begin{align*}
    (\mf{f}_F \circ \gamma \circ \mf{f}^{-1}_F)(1_{(b_{\sigma})_{\sigma}})  & =  (\mf{f}_F \circ \gamma)(\prod\limits_{\sigma' \in \Omega_E} \mf{f}^{-1}_F\sigma'(\varpi_E)^{-b_{\sigma'}})_{(b_{\ov{\mf{f}_F} \sigma})_{\sigma}})\\
    &=\mf{f}_F(\prod\limits_{\sigma' \in \Omega_E} \gamma\mf{f}^{-1}_F\sigma'(\varpi_E)^{-b_{\sigma'}})_{(b_{\ov{\mf{f}_F \gamma^{-1}} \sigma})_{\sigma}})\\
    &=(\prod\limits_{\sigma' \in \Omega_E} \mf{f}_F\gamma\mf{f}^{-1}_F\sigma'(\varpi_E)^{-b_{\sigma'}} \prod\limits_{\sigma'' \in \Omega_E} \sigma''(\varpi_E)^{b_{\ov{\mf{f}_F \gamma^{-1}\mf{f}^{-1}_F}\sigma''}}))_{(b_{\ov{\mf{f}_F \gamma^{-1} \mf{f}^{-1}_F} \sigma})_{\sigma}}\\
    &=1_{(b_{\ov{\mf{f}_F\gamma^{-1}\mf{f}^{-1}_F}\sigma})_{\sigma}},
\end{align*}
as desired.

The action does not depend on the choice of $\mf{f}_F$. Indeed, any other lift is of the form $\gamma \mf{f}_F$ for $\gamma \in \Gamma_{L_E/L}$ and we have 
\begin{equation*}
    \gamma \mf{f}_F(\ell_{(b_{\sigma})_{\sigma}}) = \gamma \mf{f}_F(\ell) \prod\limits_{\sigma' \in \Omega_E} \gamma\sigma'(\varpi_E)^{b_{\ov{\mf{f}^{-1}_F}\sigma'}}_{(b_{\ov{\mf{f}^{-1}_F\gamma^{-1}}\sigma})_{\sigma}} = \gamma\mf{f}_F(\ell) \prod\limits_{\sigma' \in \Omega_E} \sigma'(\varpi_E)^{b_{\ov{\mf{f}^{-1}_F\gamma^{-1}}\sigma'}}_{(b_{\ov{\mf{f}^{-1}_F\gamma^{-1}}\sigma})_{\sigma}},
\end{equation*}
which is the same action $\gamma\mf{f}_F$ would have if we had chosen it as our lift of Frobenius.

Finally, we need to compute the action of the canonical generator $\mf{f}_E$ of $W_{L_E/E} \subset W_{L_E/F}$. We can write $\mf{f}_E = \gamma \mf{f}^{f(E/F)}_F$ for some $\gamma \in \Gamma_{L_E/L}$. Then
\begin{align*}
    \gamma \mf{f}^{f(E/F)}_F(1_{(b_{\sigma})_{\sigma}})&=\gamma \left(\prod\limits_{k=1}^{f(E/F)} \prod\limits_{\sigma' \in \Omega_E} \mf{f}^{f(E/F) -k}_F\sigma'(\varpi_E)^{b_{\ov{\mf{f}^{-k}_F}\sigma}}_{(b_{\ov{\mf{f}^{-f(E/F)}_F}\sigma})_{\sigma}}\right)\\
    &=\gamma \left(\prod\limits^{f(E/F)}_{k=1} \prod\limits_{\sigma' \in \mf{f}^{-k}_F(\Omega_E)} \mf{f}^{f(E/F)}_F\sigma'(\varpi_E)^{b_{\sigma'}}_{(b_{\ov{\mf{f}^{-f(E/F)}_F}\sigma'})_{\sigma}}\right)\\
    &=\gamma\left(\prod\limits_{\sigma' \in \Gamma_{E/F}} \mf{f}^{f(E/F)}_F\sigma'(\varpi_E)^{b_{\sigma'}}_{(b_{\ov{\mf{f}^{-f(E/F)}_F}\sigma})_{\sigma}}\right)=\prod\limits_{\sigma' \in \Gamma_{E/F}} \sigma'(\varpi_E)^{b_{\sigma'}}_{(b_{\sigma})_{\sigma}}.
\end{align*} 

With this action in hand, we are ready to define a descent datum for our category $RP_{E/F,s}$. Fix a section $s_{E} \colon \Gamma_{E/F} \to W_{L_{E}/F}$, which determines a 2-cocycle of $\Gamma_{E/F}$ valued in $\Z$, denoted by $c_{E}$. For a fixed $\gamma \in \Gamma_{E/F}$, we define the functor $\beta_{\gamma} \colon RP_{E/F,s} \to RP_{E/F,s}$ to send the triple $(V, \Phi, \mc{D})$ to the triple $(\beta_{\gamma}(V), \beta_{\gamma}(\Phi), \beta_{\gamma}(\mc{D}))$ such that
\begin{itemize}
    \item  the underlying vector space $\beta_{\gamma}(V)$ is defined to be $V \otimes_{R_{E/F}} R_{E/F}$ where the map $R_{E/F} \to R_{E/F}$ is given by the action of $s_E(\gamma)$,
    \item the endomorphism $\beta_{\gamma}(\Phi)$ is defined by $\Phi \otimes ( s_E(\gamma) \circ \Phi_R)$,
    \item the grading $\beta_{\gamma}(\mc{D})$ on $\beta_{\gamma}(V)$ is such that $v_{(a_{\gamma \sigma}/s)_{\sigma})} \otimes r_{(b_{\sigma})_{\sigma}}$ is in the $((a_{\sigma} +sb_{\sigma})/s)_{\sigma}$ graded piece.
\end{itemize} 

A morphism $f: (V, \Phi, \mc{D}) \to (V', \Phi', \mc{D}')$  is mapped by $\beta_{\gamma}$ to $\beta_{\gamma}(f) : = f \otimes \id$.  This construction gives a $\gamma$-semilinear tensor equivalence from $RP_{E/F,s}$ to itself. For instance, for $e \in E$, we observe that
\begin{equation*}
    \beta_{\gamma}(e)(v \otimes r) = ev \otimes r=v\otimes \gamma(e)r=\gamma(e)(v \otimes r).
\end{equation*}

To finish giving our descent datum, we need to provide an isomorphism of functors $\mu_{\gamma',\gamma} \colon \beta_{\gamma'\gamma} \to \beta_{\gamma'} \circ \beta_{\gamma}$ for all pairs $(\gamma',\gamma)$ which satisfies a cocycle condition (explained in more detail later). The morphism $\mu_{\gamma',\gamma}(X)$ will be defined for any object $(V, \Phi, \mc{D})=X \in RP_{E/F,s}$ by the map
\begin{equation*}
    \mu_{\gamma', \gamma}(X): v \otimes r \mapsto \Phi^{-c_E(\gamma', \gamma)}(v) \otimes s_E(\gamma'\gamma')(r).
\end{equation*}

We check that this is a well-defined morphism in $RP_{E/F,s}$. First we note that for $v \in V$ and $r \in R_{E/F}$, we have $\mu_{\gamma', \gamma}(X)(v \otimes r)=\mu_{\gamma', \gamma}(rv \otimes 1)$. Indeed,
\begin{align*}
    \mu_{\gamma', \gamma}(X)(v \otimes r) & = \Phi^{-c_E(\gamma',\gamma)}(v) \otimes s_E(\gamma'\gamma)(r)  = \Phi^{-c_E(\gamma',\gamma)}(v) \otimes s_E(\gamma')s_E(\gamma) \Phi^{-c_E(\gamma', \gamma)}(r)\\
    & =  \Phi^{-c_E(\gamma', \gamma)}(rv) \otimes 1 =  \mu_{\gamma'\gamma}(X)(rv \otimes 1).
\end{align*}

Now, evidently $\mu_{\gamma', \gamma}(X)$ is $\Phi$-equivariant and preserves gradings, so it remains to show $R_{E/F}$-linearity. For any $v \in V$ and $r \in R_{E/F}$, we compute that 
\begin{equation*}
    \mu_{\gamma', \gamma}(X)( v \otimes r) = \Phi^{-c_E(\gamma', \gamma)}(v) \otimes s_E(\gamma',\gamma)(r) = r\cdot \mu_{\gamma'\gamma}(v \otimes 1).
\end{equation*}

We claim that the family of isomorphisms of functors $\mu_{\gamma', \gamma}$ satisfies the ``cocycle condition.'' Explicitly, for $\gamma'', \gamma', \gamma \in \Gamma_{E/F}$, we need to check that the following diagram commutes:
\[
\begin{tikzcd}
\beta_{\gamma''\gamma'\gamma}(V) \arrow["{(\mu_{{\gamma'',\gamma'\gamma}})_{V}}"]{rr} \arrow[swap, "{(\mu_{\gamma''\gamma',\gamma})_{V}}"]{dd} & & \beta_{\gamma''}(\beta_{\gamma'\gamma}(V)) \arrow["{\beta_{\gamma''}((\mu_{{\gamma',\gamma}})_{V})}"]{dd} \\
&&\\
\beta_{\gamma''\gamma'}(\beta_{\gamma}(V)) \arrow["{(\mu_{\gamma'',\gamma'})_{\beta_{\gamma}(V)}}"]{rr} && \beta_{\gamma''}(\beta_{\gamma'}(\beta_{\gamma}(V))).
\end{tikzcd}
\]

We can write any element in the above spaces in the form $v \otimes 1$. Then we note that on elements of this form, the top arrow is given by $\Phi^{-c_{E}(\gamma'',\gamma'\gamma)}$, the right arrow by $\Phi^{-c_{E}(\gamma',\gamma)}$, the left by $\Phi^{-c_{E}(\gamma''\gamma', \gamma)}$, and the bottom by $\Phi^{-c_{E}(\gamma'',\gamma')}$. The desired commutativity then follows from the fact that $c_{E}$ is a $2$-cocycle (keeping in mind that the action of $\Gamma_{E/F}$ on $\Z = W_{L_{E}/E}$ is trivial).

With this setup in hand, we get a Tannakian category over $F$, denoted by $R\mc{W}_{E/F,s}$. Explicitly, the objects of $R\mc{W}_{E/F, s}$ are tuples $(X, f_{\gamma}: X \to \beta_{\gamma}(X))$ where $X$ is an object of $RP_{E/F, s}$ and $f_{\gamma}$ is a morphism such that 
\begin{equation*}
    \mu_{\gamma', \gamma}(X) \circ f_{\gamma'\gamma} = \beta_{\gamma'}(f_{\gamma}) \circ f_{\gamma'}.
\end{equation*}
Morphisms in $R\mc{W}_{E/F, s}$ are morphisms in $RP_{E/F,s}$ that commute with the $f_{\gamma}$ maps. The band of $R\mc{W}_{E/F, s}$ is an $F$-group scheme whose base-change to $E$ is $\prod_{\Gamma_{E/F}} \mu_{s}$. In fact, we have the following claim.
\begin{claim}
The band of the gerbe associated to $R\mc{W}_{E/F,s}$ is $\Res_{E/F}(\mu_s)$.
\end{claim}
To make the verification of this claim easier, we note that we can define an analogous descent datum on the category $P_{E/F,s}$ defined in the identical way, except forgetting about any $R_{E/F}$-module structure (so for instance, $\beta_{\gamma}(V) = V \otimes_{L_E} L_E$ where the map $L_E \to L_E$ is the Galois action by $s_E(\gamma)$). This gives a Tannakian category $\mc{W}_{E/F,s}$ over $F$. Moreover, the quotient functor $q$ extends to a morphism $\tilde{q}$ of Tannakian categories with descent data; that is, for an object $(X, f_{\gamma})$ in $\mc{W}_{E/F,s}$, the object $(q(X), \tilde{q}(f_{\gamma}))$ lies in $R\mc{W}_{E/F,s}$. The map $\tilde{q}(f_{\gamma}): V \otimes R_{E/F} \to V \otimes R_{E/F} \otimes_{R_{E/F}} R_{E/F} =\beta_{\gamma}(V \otimes R_{E/F})$ is defined by $v \otimes r \mapsto f_{\gamma}(v) \otimes s_E(\gamma)^{-1}(r) \otimes 1$. This map is $R_{E/F}$-linear since for $r \in R_{E/F}$ and $v \in V$, we have
\begin{equation*}
    v \otimes r \mapsto f_{\gamma}(v) \otimes s_E(\gamma)^{-1}(r) \otimes 1 = f_{\gamma}(v) \otimes 1 \otimes r = r \cdot q(f_{\gamma})(v \otimes 1).
\end{equation*}

This extension of $q$ is again a quotient functor, and the same argument as in \S \ref{sec:mod} shows that the band of $R\mc{W}_{E/F,s}$ is given by the $s$-torsion subgroup-scheme of the $F$-group scheme giving the band of $\mc{W}_{E/F,s}$. It thus suffices to show that the latter band is the $F$-group scheme $\text{Res}_{E/F}(\mathbb{G}_{m})$.

Since we have identified the band of $P_{E/F,s}$ with the $F$-group scheme $\prod_{\Gamma_{E/F}} \mathbb{G}_{m}$ via the isomorphism of Tannakian categories over $E$ with the category $\text{Rep}_{E_{s}}(\mc{E}_{(c_{\sigma})_{\sigma}})$ constructed in \S \ref{sec:prod}, we will work with this latter category. Conjugating the functors $\beta_{\gamma}$ with this ($E$-linear) identification gives a descent datum on $\text{Rep}_{E_{s}}(\mc{E}_{(c_{\sigma})_{\sigma}})$. For $\gamma \in \Gamma_{E/F}$, this is given by taking the representation $(V, \rho)$ to the representation $(V', \rho')$ where $V'=V \otimes_{E_s} E_s$ such that the map $E_s \to E_s$ is given by the action of $s_{E}(\gamma)$ restricted to $E_{s}$ and $V'_{(h_{\sigma})_{\sigma}} = V_{(h_{\gamma \sigma})_{\sigma}}$. We define $\rho'$ such that the element $((1),\mf{f}_{E})$ acts by $\Phi \otimes \mf{f}_E$ where the action on the right is the Galois action, and an element $(a_{\sigma})_{\sigma}$ of the torus acts on $V_{(h_{\sigma})}$ by the scalar $\prod_{\sigma \in \Gamma_{E/F}} a_{\sigma}^{h_{\gamma^{-1}\sigma}}$. With the descent functors $\beta'_{\gamma}$ on $\text{Rep}_{E_{s}}(\mc{E}_{(c_{\sigma})_{\sigma}})$ in hand, we are ready to compute the $\Gamma_{E/F}$-action giving the band the structure of an $F$-group scheme. 

We know that the automorphisms of the canonical fiber functor $\omega$ are parametrized by $\prod_{\Gamma_{E/F}} \mathbb{G}_{m}(E_{s})$, where the tuple $(c_{\tau})$ acts on $V$ via scaling $V_{(h_{\sigma})}$ by $\prod\limits_{\tau \in \Gamma_{E/F}} c_{\tau}^{h_{\tau}}$. Starting with such an automorphism $\varphi$ and object $(V, \rho)$, we first apply $\beta_{\gamma^{-1}}$, obtaining $(V', \rho')$, and obtain an automorphism of $V'$ denoted $\varphi(\beta_{\gamma^{-1}}(V',\rho'))$. We then twist the $L_{E}$-action by $s(\gamma)^{-1}$, obtaining an automorphism $s(\gamma)^{-1}[\varphi(\beta_{\gamma^{-1}}(V',\rho'))]$ of the vector space $(V^{s(\gamma^{-1})^{-1}})^{s(\gamma)^{-1}}$, and then apply $\omega(\mu_{\gamma,\gamma^{-1}})$ to get an automorphism of $V = \omega((V,\rho))$, which we want to explicitly describe. The automorphism $\varphi(\beta_{\gamma^{-1}}(V',\rho'))$ scales $V_{(h_{\sigma})} = V'_{(h_{\gamma\sigma})}$ by $\prod_{\tau} c_{\tau}^{h_{\gamma \tau}}$ (using the twisted $L_{E}$-action), which means that $s(\gamma)^{-1}(\varphi(\beta_{\gamma^{-1}}(V',\rho')))$ scales $V_{(h_{\sigma})}$ by $s(\gamma) \cdot (\prod_{\tau} c_{\tau}^{h_{\gamma \tau}})$ (using the usual $L_{E}$-action), so that the new automorphism of our fiber functor is given by the tuple $(s(\gamma)\Phi^{c_{E}(\gamma,\gamma^{-1})}c_{\gamma^{-1}\tau})_{\tau}$. It follows that when $(c_{\tau})_{\tau} \in \prod_{\Gamma_{E/F}} \mathbb{G}_{m}(E)$, this action agrees with the usual $\Gamma_{E/F}$-action on $\text{Res}_{E/F}(\mathbb{G}_{m})(E)$. This implies that the $F$-group scheme structure of our band is given by $\text{Res}_{E/F}(\mathbb{G}_{m})$, as desired.

Finally, note that this identification of the bands of $\mc{W}_{E/F,s}$ and $R\mc{W}_{E/F,s}$ with these $F$-group schemes is compatible with our transition maps, in the sense that, at the level of bands, the functor $\iota_{s,s'}$ induces the $s'/s$-power map from $\text{Res}_{E/F}(\mathbb{G}_{m})$ to itself (and thus, the $s'/s$-power map from $\text{Res}_{E/F}(\mu_{s'})$ to $\text{Res}_{E/F}(\mu_{s})$). This fact follows from base-changing to $E$, where the analogous statement holds for the category $P_{E/F,s}$ (and thus for $RP_{E/F,s}$), as discussed in \S \ref{sec:prod}). If $(X, f_{\gamma})$ is an object of $\mc{W}_{E/F,s}$, then $(\iota(X), \iota(f_{\gamma}))$ is an object of $\mc{W}_{E/F,s'}$, similarly for $R\mc{W}_{E/F,s}$, and we define $\mc{W}_{E/F}$ and $R\mc{W}_{E/F}$ to be their direct limits, whose bands have an explicit description via the above observation.

\subsection{Concrete Tannakian categories}
We will now describe the Galois descent constructed in the previous subsection in terms of a more concrete Tannakian category over $F$.

First we define the category $\mc{I}_{E/F,s}$ for a fixed $s \in \mathbb{N}$; its objects are triples $(V, \tilde{\Phi},\mathcal{D})$ where $V$ is a finite dimensional $L_E$ vector space, $\tilde{\Phi}$ is a $\Gamma_{L_E/F}$-semilinear action of $W_{L_E/F}$ on $V$, and $\mathcal{D}$ is a $\prod_{\Gamma_{E/F}} \frac{1}{s}\mathbb{Z}$-grading of $V$, such that $V_{(a_{\sigma/s}){\sigma}}$ has an $L_{E}$-basis on which $\mf{f}_{E}^{s} \in W_{L_{E}/F}$ acts by $\prod\limits_{\sigma \in \Gamma_{E/F}} \sigma(\varpi_E)^{a_{\sigma}}$, and that for $x \in \Gamma_{L_E/F}$ lifting $\bar{x} \in \Gamma_{E/F}$, we have that $\tilde{\Phi}(x)(V_{(a_{\sigma})_{\sigma}}) = V_{(a_{\bar{x}^{-1}\sigma})_{\sigma}}$. The category $\mc{I}_{E/F,s}$ has a natural tensor structure defined such that $(V, \tilde{\Phi}) \otimes (V', \tilde{\Phi}') = (V \otimes V', \tilde{\Phi} \otimes \tilde{\Phi}')$ where $\tilde{\Phi} \times \tilde{\Phi'}$ is given explicitly by 
\begin{equation*}
    (\tilde{\Phi} \otimes \tilde{\Phi}')(\sigma)(v \otimes v') = \tilde{\Phi}(\sigma)(v) \otimes \tilde{\Phi}'(\sigma)(v').
\end{equation*}
The identity object $\mb{1}$ of $\mc{I}_{E/F}$ consists of the $L_E$ vector space $L_E$ with $\tilde{\Phi}$ equal to the normal Galois action. It is easy to check that the endomorphisms of $\mb{1}$ are in canonical bijection with $F$.

We now construct a functor $T: \mc{W}_{E/F,s} \to \mc{I}_{E/F,s}$. Given an object $(X, X \xrightarrow{f_{\gamma}} \beta_{\gamma}(X))$ of $\mc{W}_{E/F,s}$, we let the vector space $V$ in the $T$-image of $(X, X \xrightarrow{f_{\gamma}} \beta_{\gamma}(X))$ be the underlying vector space of $X$ and we let the grading $\mc{D}$ be the grading on $X$. We define $\tilde{\Phi}$ as follows. For any $\sigma \in W_{L_E/F}$ we can write uniquely $\sigma= \mf{f}_{E}^n s_E(\gamma)$ for some $\gamma \in \Gamma_{E/F}$. Then we define 
\begin{equation*}
    \tilde{\Phi}(\sigma)(v) =\Phi^n(f^{-1}_{\gamma}(v \otimes 1)).
\end{equation*}
We note that $f^{-1}_{\gamma}$ is $s_E(\gamma)$ semilinear on $V$ by definition.

To show that $\tilde{\Phi}$ is indeed a $\Gamma_{L_E/F}$-action, we need to verify for $\sigma_1, \sigma_2 \in \Gamma_{L_E/F}$ that $\tilde{\Phi}(\sigma_1) \circ \tilde{\Phi}(\sigma_2) = \tilde{\Phi}(\sigma_1\sigma_2)$. Explicitly, if $\sigma_i = \mf{f}_{E}^{n_i} s_E(\gamma_i)$, then we must show that
\begin{equation*}
    (\Phi^{n_1} \circ f^{-1}_{\gamma_1}) \circ (\Phi^{n_2} \circ f^{-1}_{\gamma_2}) = \Phi^{n_1+n_2 + c_E(\gamma_1,\gamma_2)}f^{-1}_{\gamma_1\gamma_2}. 
\end{equation*}
By construction, we have
\begin{equation*}
    f^{-1}_{\gamma_1\gamma_2} = f^{-1}_{\gamma_1} \circ f^{-1}_{\gamma_2} \circ \Phi^{-c_E(\gamma_1,\gamma_2)},
\end{equation*}
which implies the desired equality. 

We also need to check that $\tilde{\Phi}$ acts on the grading correctly. Let $\gamma \in \Gamma_{E/F}$. Then we have
\begin{equation*}
\tilde{\Phi}(s_E(\gamma))(V_{(a_{\sigma}/s)_{\sigma}}) = f^{-1}_{\gamma}(V_{(a_{\sigma}/s)_{\sigma}}\otimes 1)=f^{-1}_{\gamma}(\beta_{\gamma}(V)_{(a_{\gamma^{-1}\sigma}/s)_{\sigma}})=V_{(a_{\gamma^{-1}\sigma}/s)_{\sigma}}.
\end{equation*}
Finally, $T$ is a functor because morphisms of $\mc{W}_{E/F,s}$ must necessarily commute with the $\Phi$ and $f_{\gamma}$ maps.

We now define a functor $S: \mc{I}_{E/F, s} \to \mc{W}_{E/F, s}$, giving a quasi-inverse of $T$. Given an object $(V, \tilde{\Phi}, \mc{D})$ of $\mc{I}_{E/F, s}$, we define its image under $S$ to have underlying vector space equal to $V$ and $\prod\limits_{\Gamma_{E/F}} \frac{1}{s} \Z$-grading $\mc{D}$. We define $\Phi$ to be $\tilde{\Phi}(\mf{f}_E)$. Finally, for our fixed section $s_E: \Gamma_{E/F} \to W_{L_E/F}$, we define $f_{\gamma}: V \to \beta_{\gamma}(V)$ by
\begin{equation*}
    f_{\gamma}: v \mapsto \tilde{\Phi}(s_E(\gamma))^{-1}(v) \otimes 1.
\end{equation*}

There are many things we need to check in the above definition. First, we check that $f_{\gamma}$ is an isomorphism of $\mf{f}_{E}$-spaces from $V$ to $\beta_{\gamma}(V)$ for a fixed $\gamma$. It is evidently an isomorphism of abelian groups, and is $\mf{f}_{E}$-equivariant, since we have, for $v \in V$, the identity $$\tilde{\Phi}(s_{E}(\gamma))^{-1}(\mf{f}_{E} \cdot v) \otimes 1 = \tilde{\Phi}(s_{E}(\gamma)^{-1})(\tilde{\Phi}(\mf{f}_{E})(v)) \otimes 1 = \tilde\Phi(s_{E}(\gamma)^{-1}\mf{f}_{E})(v) \otimes 1 $$
$$= \tilde{\Phi}(\mf{f}_{E} s_{E}(\gamma)^{-1})(v) \otimes 1 = \mf{f}_{E} \cdot (\tilde{\Phi}(s_{E}(\gamma)^{-1})(v)) \otimes 1 = \mf{f}_{E} \cdot (\tilde{\Phi}(s_{E}(\gamma))^{-1}(v)) \otimes 1 ,$$ where we have used that $\mf{f}_{E}$ is contained in $Z(W_{L_{E}/F})$. We also have $L_{E}$-linearity due to the identity, for $\ell \in L_{E}$, given by $$\tilde{\Phi}(s_{E}(\gamma)^{-1})(v \otimes \ell) = s_{E}(\gamma)^{-1}(\ell)\tilde{\Phi}(s_{E}(\gamma)^{-1})(v) \otimes 1 = \tilde{\Phi}(s_{E}(\gamma)^{-1})(v) \otimes \ell.$$ 

We now check that the maps $f_{\gamma}$ satisfy the cocycle condition with respect to $\beta_{\gamma}$ and $\mu_{\gamma',\gamma}$ stated in the previous subsection. Recall that this condition is given by the equation
\begin{equation*}
    (\mu_{\gamma',\gamma})(X) \circ f_{\gamma'\gamma} = \beta_{\gamma'}(f_{\gamma}) \circ f_{\gamma'}
\end{equation*}
for any $X \in \mc{W}_{E/F,s}$. In our setting, this equation is 
$$\tilde{\Phi}(c_{E}(\gamma',\gamma)^{-1}) \circ \tilde{\Phi}(s_{E}(\gamma'\gamma))^{-1} = \tilde{\Phi}(s_{E}(\gamma))^{-1} \circ \tilde{\Phi}(s_{E}(\gamma'))^{-1},$$ which is trivial from the definition of the cocycle $c_{E}$, since $\tilde{\Phi}$ is a group action. It follows that the above functor gives a well-defined object of $\mc{W}_{E/F,s}$. Any morphism in $\mc{I}_{E/F,s}$ gives a morphism in $\mc{W}_{E/F,s}$, since morphisms in $\mc{I}_{E/F,s}$ preserve the grading by construction and commute with the $W_{L_{E}/F}$-action, and therefore with the maps $f_{\gamma}$. Both functors $S$ and $T$ are evidently $F$-linear tensor functors. We thus immediately obtain:

\begin{cor} The gerbe over $(\text{Sch}/F)_{\text{fpqc}}$ corresponding to the Tannakian category $\mc{I}_{E/F,s}$ over $F$ has band which is isomorphic as an $F$-group scheme to $\text{Res}_{E/F}(\mathbb{G}_{m}).$
\end{cor}

It is not hard to see how this generalizes to give a more concrete interpretation of the category $R\mc{W}_{E/F,s}$. Recall the $W_{L_{E}/F}$-action $\tilde{\Phi}_{R_{E/F}}$ on the ring $R_{E/F}$ defined in the previous subsection. We now define the category $R\mc{I}_{E/F,s}$, whose objects consist of tuples $(V, \mc{D},\tilde{\Phi})$ where $V$ is an $L_E$-vector space and $\mc{D}$ is a $\prod\limits_{\Gamma_{E/F}} \frac{1}{s}\Z$-grading on $V$. The graded vector space $V$ is equipped with an action of the graded ring $R_{E/F}$ and we require that $V$ is finitely generated as an $R_{E/F}$-module. Finally, $\tilde{\Phi}$ is a semilinear action of $W_{L_E/F}$ on $V$ such that $\tilde{\Phi}(\gamma)(rv)=\tilde{\Phi}_{R_{E/F}}(\gamma)(r) \tilde{\Phi}(\gamma)(v)$ for $r \in R_{E/F}$, and such that $V_{(b_{\sigma}/s)_{\sigma}}$ has an $L_E$-basis upon which $\tilde{\Phi}(\mf{f}^s_E)$ acts by scaling by $\prod\limits_{\sigma \in \Gamma_{E/F}} \sigma(\varpi_E)^{b_{\sigma}}$. To simplify notation, we will often denote $\tilde{\Phi}(\gamma)(v)$ simply by $\gamma \cdot v$ or $\gamma v$. The category $R\mc{I}_{E/F, s}$ is Tannakian over $F$ with identity object $R_{E/F}$. An identical argument to the one above proves that the functors $T$ and $S$ may be extended to inverse functors $\tilde{S}, \tilde{T}$ between $R\mc{I}_{E/F,s}$ and $R\mc{W}_{E/F,s}$, showing that the former is banded by the $F$-group scheme $\text{Res}_{E/F}(\mu_{s})$. By ``extended," we mean in the sense that the diagram of quotient functors
\[
\begin{tikzcd}
\mc{W}_{E/F,s} \arrow{d} \arrow{r} & \mc{I}_{E/F,s} \arrow{d} \\
R\mc{W}_{E/F,s} \arrow{r} & R \mc{I}_{E/F,s},
\end{tikzcd}
\]
where the left vertical arrow the quotient functor discussed in the previous subsection, and the right vertical arrow is the functor which sends the object $(V, \mc{D}, \tilde{\Phi})$ to the object $(V \otimes_{L_{E}} R_{E/F}, \mc{D}', \tilde{\Phi} \otimes \tilde{\Phi}_{R_{E/F}})$, where $\mc{D}'$ is the grading induced by tensoring the grading of $V$ with that of $R_{E/F}$. For $s'/s$, there is a clear transition functor from $\mc{I}_{E/F,s}$ to $\mc{I}_{E/F,s'}$, same for $R\mc{I}_{E/F,s}$, and one checks that the above equivalences with $\mc{W}_{E/F,s}$ and $R\mc{W}_{E/F,s}$ (which also have transition maps) are compatible as we vary $s$, giving an identification of $\mc{W}_{E/F}$ and $R\mc{W}_{E/F}$ with the direct limit of the systems categories, denoted by $\mc{I}_{E/F}$ and $R\mc{I}_{E/F}$, respectively. However, when constructing our final Tannakian category, we will use slightly more interesting transition maps than these (in particular, we will also change the extension $E/F$).

\subsection{Transition maps}
Consider a tower of finite Galois extensions $K/E/F$. Fix an anti-uniformizer $\varpi_E$ of $E$. Let $K' \subset K$ be the maximal unramified extension of $E$. Then $K/K'$ is totally ramified and we have our fixed anti-uniformizer $\varpi_K$ such that $N_{K/K'}(\varpi_K)=\varpi_E$. We define a transition functor $p_{K/E}: R\mc{I}_{E/F, s} \to R\mc{I}_{K/F, s}$ defined according to our choices of uniformizers. We pick a lift of the Frobenius of $F$ in the absolute Weil group $W_{\ov{F}/F}$, whose restriction we use in our definitions of the actions of $W_{L_{K}/F}$ and $W_{L_{E}/F}$ on $R_{K/F}$ and $R_{E/F}$, respectively. To begin, we consider the $\prod\limits_{\tau \in \Gamma_{K/F}} \Z$-graded ring $R'_{E/F, K}$ whose $(c_{\tau})_{\tau}$-graded piece is $0$ if the grading is not fixed under the $\Gamma_{K/E}$-action and is the $(c_{\ov{\tau}})_{\sigma}$-graded piece of $R_{E/F}$ otherwise. We then define $R_{E/F, K} := R'_{E/F, K} \otimes_{L_E} L_K$. We give $R_{E/F, K}$ the diagonal $W_{L_K/F}$-action whereby $W_{L_K/F}$ acts on the first factor by projection to $W_{L_E/F}$ and on the second factor by the natural Galois action. Note that we have an inclusion of graded rings 
\begin{equation*}
    R_{E/F, K} \hookrightarrow R_{K/F}.
\end{equation*}
Moreover, we claim this inclusion respects the $W_{L_K/F}$-actions. We need only check that the actions of $\mf{f}_{F,K}$ (our chosen lift of Frobenius) agree on $1_{(c_{\tau})_{\tau}}$ since the actions of $\Gamma_{L_K/L}$ both act by the natural Galois action and shift the gradings in the same way. We have 
$$
  \mf{f}_{F,K}(1_{(c_{\tau})_{\tau}}) = \prod\limits_{\tau' \in \Omega_{K}} \tau'(\varpi_K)^{c_{\mf{f}_{F,K}^{-1}\tau'}}_{(c_{\mf{f}_{F,K}^{-1}\tau})_{\tau}} = \prod\limits_{\sigma' \in \Gamma_{K'/(K \cap L)}} \prod\limits_{\tau'' \in \Gamma_{K/K'}} \sigma'\tau''(\varpi_K)^{c_{\mf{f}_{K,F}^{-1}\tau'}}_{(c_{\mf{f}_{K,F}^{-1}\tau})_{\tau}}$$
  $$= \prod\limits_{\sigma' \in \Gamma_{K'/(K \cap L)}} \sigma'(\varpi_E)^{c_{\mf{f}_{F,E}^{-1}\sigma'}}_{(c_{\mf{f}_{F,E}^{-1}\sigma})_{\sigma}} = \prod\limits_{\sigma' \in \Gamma_{E/(E \cap L)}} \sigma'(\varpi_E)^{c_{\mf{f}_{F,E}^{-1}\sigma'}}_{(c_{\mf{f}_{F,E}^{-1}\sigma})_{\sigma}} = [\mf{f}_{F,E}1_{(c_{\sigma})_{\sigma}}] \otimes 1 = \mf{f}_{F,K}(1_{(c_{\sigma})_{\sigma}} \otimes 1),$$

To obtain the fourth equality, we have used the fact that the restriction-to-$E$ map induces a surjection $\Gamma_{K'/(K \cap L)} \twoheadrightarrow \Gamma_{E/(E \cap L)}$, with kernel equal to $\Gamma_{K'/[E \cdot (K \cap L)]}$, where now $E \cdot (K \cap L) = K'$, since the maximal unramified subextension of $K/E$ is the compositum of $E$ and the maximal unramified subextension of $K/F$, by elementary algebraic number theory. The last equality follows from the fact that, by construction, the image of $\mf{f}_{F,K}$ in $W_{L_{E}/F}$ is $\mf{f}_{F,E}$.

We are now ready to define the functor $p_{K/E}$. For an object $M$ of $R\mc{I}_{E/F, s}$, we define $M_K$ to be the $\prod\limits_{\tau \in \Gamma_{K/F}} \frac{1}{s} \Z$-graded $R_{E/F, K}$-module whose $(c_{\tau})_{\tau}$-graded piece is $0$ if the grading is not fixed under the $\Gamma_{K/E}$-action and is the $(c_{\ov{\tau}})_{\ov{\tau}}$ graded piece of $M \otimes_{L_{E}} L_K$ otherwise. We then define $p_{K/E}(M) := M_K \otimes_{R_{E/F, K}} R_{K/F}$ (with the induced grading) and define the $W_{L_K/F}$-action diagonally, noting this is well-defined by the $W_{L_K/F}$-equivariance of $R_{E/F, K} \hookrightarrow R_{K/F}$.

Note also that there is a similar transition functor $p'_{K/E} \colon \mc{I}_{K/F,s} \to \mc{I}_{E/F,s}$, defined simply by sending the $W_{L_{E}/F}$-module $M$ to the $W_{L_{K}/F}$-module $M \otimes_{L_{E}} L_{K}$ with diagonal $W_{L_{K}/F}$-action induced on the first factor by projection to $W_{L_{E}/F}$ and grading given the same way as with the functor $p_{K/E}$. Note that, trivially, the transition functors $p'_{K/E}$ and $p_{K/E}$ are compatible with the quotient maps $\mc{I}_{?/F,s} \to R\mc{I}_{?/F,s}$, where $?=E$ or $K$. By abuse of notation, we will also use $p'_{K/E}$ to denote the functor from $P_{K/F,s}$ to $P_{E/F,s}$ given in the same way as $p'_{K/E}$, except now with diagonal $\Gamma_{L_{K}/K}$-action on $M \otimes L_{K}$ induced on the first factor via projecting into $\Gamma_{L_{E}/E}$. 

As usual, to describe what $p_{K/E}$ does on bands, we will first describe what $p'_{K/E}$ does, which in turn we will do via our identifications of $P_{E/F,s}$ and $P_{K/F,s}$ with the categories $\mc{E}_{(c_{\sigma})_{\sigma}}$ and $\mc{E}_{(d_{\sigma})_{\sigma}}$, respectively. Here, as in \S \ref{sec:prod}, the tuples $(c_{\sigma})_{\sigma}$ and $(d_{\sigma})_{\sigma}$ are the representatives of the anti-canonical class in $H^{2}(\Gamma_{E_{s}/E}, E_{s}^{\times})$ and $H^{2}(\Gamma_{K_{s}/K}, K_{s}^{\times})$, respectively corresponding to $\varpi_{E}$ and $\varpi_{K}$. Conjugating $p_{K/E}'$ by these equivalences gives a functor from $\text{Rep}_{E_{s}}(\mc{E}_{(c_{\sigma})_{\sigma}})$ to $\text{Rep}_{K_{s}}(\mc{E}_{(d_{\sigma})_{\sigma}})$, which we will explicitly compute. 

Recall that, given an $L_{E}$-representation $(V, \rho)$ of the group $\mc{E}_{(c_{\sigma})_{\sigma}}$, determined by the homomorphism $h_{\rho} \colon \prod_{\Gamma_{E/F}} \mathbb{G}_{m}(L_{E}) \to \mathrm{GL}(V)$ and the $\mf{f}_{E}$-semilinear automorphism induced by $((1),\mf{f}_{E})$, we obtain an object $(V', \Phi, \mc{D})$ of $P_{E/F,s}$ as in \S \ref{sec:prod}, with $V' = V \otimes_{E_{s}} L_{E}$, $\Phi$ given by the action $((1),\mf{f}_{E})$, and $V'_{(h_{\sigma}/s)} = V_{(h_{\sigma})} \otimes_{E_{s}} L_{E}$. Then, by definition, $p'_{K/E}$ sends this to the object $(V' \otimes_{L_{E}} L_{K}, \Phi', \mc{D}')$, where $\Phi'$ acts diagonally on $ V' \otimes_{L_{E}} L_{K}$ by $\Phi^{f(K/E)} \otimes \mf{f}_{K}$ and $V''_{(h_{\tau}/s)} = 0$ if $(h_{\tau}/s)$ is not $\Gamma_{K/E}$-fixed, and equals $V'_{(h_{\bar{\tau}}/s)} \otimes_{L_{E}} L_{K}$ otherwise. To compute the image of this object under the equivalence, we note that the subspace $$V''_{(h_{\tau}/s)} = V'_{(h_{\bar{\tau}}/s)} \otimes_{L_{E}} L_{K} = V_{(h_{\bar{\tau}})} \otimes_{E_{s}} L_{K} = (V_{(h_{\bar{\tau}})} \otimes_{E_{s}} K_{s}) \otimes_{K_{s}} L_{K}$$ has an $L_{K}$-basis given by $\{e_{i} \otimes 1 \otimes 1\}$, where $e_{i}$ is an $E_{s}$-basis of $V_{(h_{\bar{\tau}})}$ on which $\Phi^{s}$ acts by $\prod_{\sigma \in \Gamma_{E/F}} \sigma(\varpi_{E})^{h_{\sigma}}$. Thus, $$(\Phi')^{s}(e_{i} \otimes 1 \otimes 1) =[\prod_{\sigma \in \Gamma_{E/F}} \sigma(\varpi_{E})^{f(K/E)h_{\sigma}}] (e_{i} \otimes 1 \otimes 1),$$ which, by construction, equals $$[\prod_{\sigma \in \Gamma_{E/F}} \sigma(\prod_{\gamma \in \Gamma_{K/K'}} \gamma(\varpi_{K}))^{f(K/E)h_{\sigma}}](e_{i} \otimes 1 \otimes 1).$$ Finally, the above expression may be rewritten as $$[\prod_{\tau \in \Gamma_{K/F}} \tau(\varpi_{K})^{h_{\bar{\tau}}}] (e_{i} \otimes 1 \otimes 1),$$ showing that the corresponding representation of $\mc{E}_{(d_{\sigma})}$, denoted by $(W, \rho')$, has $W_{(h_{\tau})} = 0$ if $(h_{\tau})$ is not $\Gamma_{K/E}$-stable, and is $V_{(h_{\bar{\tau}})} \otimes_{E_{s}} K_{s}$ otherwise. Thus, if the automorphism of the fiber functor $\omega_{K}$ acts on $W_{(h'_{\tau})}$ via scaling by $\prod_{\tau \in \Gamma_{K/F}} a_{\tau}^{h'_{\tau}}$, (where $(a_{\tau}) \in \prod_{\Gamma_{K/F}} \mathbb{G}_{m}(K_{s})$), then the above calculation shows that we get an induced automorphism of $(\omega_{E})_{K}$, which acts on $V_{(h_{\sigma})} \otimes_{E_{s}} K_{s}$ by $\prod_{\sigma \in \Gamma_{E/F}} (\prod_{\tau \in \Gamma_{K/F}, \tau \mapsto \sigma} a_{\tau})^{h_{\sigma}}$, which shows that the induced map, on $K_{s}$-points, is given by the ``norm" map $\prod_{\Gamma_{K/F}} \mathbb{G}_{m}(K_{s}) \to \prod_{\Gamma_{E/F}} \mathbb{G}_{m}(K_{s})$. 

By construction, the $F$-group morphism on bands from $\text{Res}_{K/F}(\mathbb{G}_{m})$ to $\text{Res}_{E/F}(\mathbb{G}_{m})$ induced by $p'_{K/E}$ is given on $K_{s}$-points via the map described explicitly at the bottom of the previous paragraph; the same is true of the canonical ``norm map" $\text{Res}_{K/F}(\mathbb{G}_{m}) \to \text{Res}_{E/F}(\mathbb{G}_{m})$. Since both schemes are tori and $K_{s}$ is infinite, by unirationality this implies that the two maps coincide. The following diagram of transition maps commutes
\[
\begin{tikzcd}
\mc{I}_{E/F,s} \arrow{d} \arrow{r} & \mc{I}_{K/F,s} \arrow{d} \\
R\mc{I}_{E/F,s} \arrow{r} & R \mc{I}_{K/F,s},
\end{tikzcd}
\]
where the vertical arrows are the quotient maps discussed above, which means that the morphism of $F$-group schemes $\text{Res}_{K/F}(\mu_{s}) \to \text{Res}_{E/F}(\mu_{s})$ associated to $p_{K/E}$ is induced by the inclusion $\text{Res}_{K/F}(\mu_{s}) \hookrightarrow \text{Res}_{K/F}(\mathbb{G}_{m})$ followed by the map induced by $p'_{K/E}$, which implies that it is also given by the norm map. Finally, for $s \mid s'$, we define the transition map $p_{K/E,s',s} \colon \mc{I}_{E/F,s} \to \mc{I}_{K/F,s'}$ to be the composition 
$$\mc{I}_{E/F,s} \xrightarrow{p_{K/E}} \mc{I}_{K/F,s} \xrightarrow{\iota_{s,s'}} \mc{I}_{K/F,s'},$$ which on bands gives the map $\text{Res}_{K/F}(\mu_{s'}) \to \text{Res}_{E/F}(\mu_{s})$ of $F$-group schemes given by $N_{K/F}$ composed with the $s'/s$-power map.

\subsection{The category $\RigIsoc_{F}$}\label{sec:RigIsocdef}
We are now ready to define the Tannakian category over $F$ which parametrizes rigid inner forms. We define the category $\RigIsoc_{F}$ to be the direct limit of the categories $R\mc{I}_{E/F,s}$ over the inductive system of tuples $(E, s)$, where $E/F$ is a finite Galois extension and $s \in \mathbb{N}$, which is ordered by the relation that $(K,s') > (E,s)$ if and only if $K$ contains $E$ and $s \mid s'$, with respect to the transition maps $p_{K/E,s,s'}$ defined in the previous subsection. It will frequently be convenient to fix a cofinal system of objects in the indexing category indexed by $\mathbb{N}$, denoted by $(E_{k}, n_{k})$. Note that $\varinjlim_{k} R\mc{I}_{k} = \RigIsoc_{F}$, where $R\mc{I}_{k} := R\mc{I}_{E_{k}/F,n_{k}}$.

By construction, $\RigIsoc_{F}$ is a Tannakian category over $F$, with fiber functor over $\ov{L}$, and the corresponding gerbe is banded by the $F$-group scheme $u = \varprojlim_{k} \text{Res}_{E_{k}/F} (\mu_{n_{k}})$. Note that each category $R\mc{I}_{E/F,s}$ has a fiber functor over the finite extension $E_{s}/F$, given explicitly by sending $(V, \mc{D}, \Phi)$ to the (finite-dimensional) $E_{s}$-vector space 
\begin{equation}\label{keyfiberfunctor}
\bigoplus_{0 \leq (a_{\sigma}/s) < 1} ([\bigoplus_{(b_{\sigma}) \in \prod_{\Gamma_{E/F}} \Z} V^{\mf{f}_{E}^{s} \cdot (\prod \sigma(\varpi_{E})^{-[a_{\sigma}+sb_{\sigma}]})}] \otimes_{R_{E/F}} R_{E/F}/I),
\end{equation}
where $V^{\mf{f}_{E}^{s} \cdot (\prod \sigma(\varpi_{E})^{-[a_{\sigma}+sb_{\sigma}]})}$ denotes the subgroup of $V$ of all elements on which $\mf{f}_{E}^{s}$ acts by the scalar $\prod_{\sigma \in \Gamma_{E/F}}\sigma(\varpi_{E})^{-[a_{\sigma}+sb_{\sigma}]}$. Note that the above expression is well-defined, since the internal direct sum is $R_{E/F}$-stable. We may extend this to a fiber functor over $\ov{F}$ by simply post-composing the above fiber functor over $E_{s}$ with $- \otimes_{E_{s}} \ov{F}$. By considering the automorphisms of this fiber functor over $\ov{F}$, we get a canonical cohomology class in the group $H^{2}(\Gamma_{\ov{F}/F}, \text{Res}_{E/F}(\mu_{s})(\ov{F}))$, which we now compute. Since $R\mc{I}_{E/F,s}$ is a quotient of $\mc{I}_{E/F,s}$, the above cohomology class maps to the class corresponding to the group of automorphisms of the fiber functor $\mc{I}_{E/F,s} \to \text{Vect}_{\ov{F}}$ given explicitly by tensoring the $E_{s}$-fiber functor $$(V, \mc{D}, \tilde{\Phi}) \mapsto \bigoplus_{(a_{\sigma}/s) \in \prod_{\Gamma_{E/F}} \frac{1}{s}\Z} V^{\mf{f}_{E}^{s} \cdot(\prod \sigma(\varpi_{E})^{-a_{\sigma}})}$$ with $\ov{F}$ over $E_{s}$, under the canonical map $H^{2}(\Gamma_{\ov{F}/F}, \text{Res}_{E/F}(\mu_{s})(\ov{F})) \to H^{2}(\Gamma_{\ov{F}/F}, \text{Res}_{E/F}(\mathbb{G}_{m})(\ov{F}))$, noting that the above fiber functors are compatible with the quotient functor from $\mc{I}_{E/F,s}$ to $R\mc{I}_{E/F,s}$, and using \cite[\S 2.10]{MilneQuot}.

Because of the above observation, it will be useful to first describe the corresponding cohomology class in $H^{2}(\Gamma_{\ov{F}/F}, \text{Res}_{E/F}(\mathbb{G}_{m})(\ov{F}))$ corresponding to $\mc{I}_{E/F,s}$. Since the Tannakian category $\mc{I}_{E/F,s}$ over $F$ has base-change to $E$ which is isomorphic as Tannakian categories over $E$ to $P_{E/F,s}$, it follows that the image of this class in $H^{2}(E, \text{Res}_{E/F}(\mathbb{G}_{m})(\ov{F}))$ is the $\Gamma_{E/F}$-fold product of the classes $-1/s \in \Q/\Z$ in $\prod_{\Gamma_{E/F}} H^{2}(E, \mathbb{G}_{m}(\ov{F}))$. In particular, the image of this class under the Shapiro isomorphism $$H^{2}(\Gamma_{\ov{F}/F}, \text{Res}_{E/F}(\mathbb{G}_{m})(\ov{F})) \xrightarrow{\sim} H^{2}(\Gamma_{\ov{F}/E}, \mathbb{G}_{m}(\ov{F}))$$ is the class mapping to $-1/s \in \mathbb{Q}/\Z$ under the local invariant map.

Recall that $H^{2}(\Gamma_{\ov{F}/F}, u(\ov{F})) = \varprojlim_{k} H^{2}(\Gamma_{\ov{F}/F}, u_{k}(\ov{F}))$, where $u_{k} := \text{Res}_{E_{k}/F}(\mu_{n_{k}})$. Using Hilbert 90 and Shapiro's Lemma, we see that the latter groups fit into the exact sequence $$0 \to H^{2}(\Gamma_{\ov{F}/F}, u_{k}(\ov{F})) \to H^{2}(\Gamma_{\ov{F}/F}, \text{Res}_{E_{k}/F}(\mathbb{G}_{m})(\ov{F})) \xrightarrow{x \mapsto n_{k}x} H^{2}(\Gamma_{\ov{F}/F}, \text{Res}_{E_{k}/F}(\mathbb{G}_{m})(\ov{F})),$$ and so we get, via Shapiro's Lemma and the local invariant map, a canonical identification of the cohomology group $H^{2}(\Gamma_{\ov{F}/F}, u_{k}(\ov{F}))$ with $-\frac{1}{n_{k}}\Z/\Z$. Moreover, it follows from our above discussion that the gerbe corresponding to the Tannakian category $R\mc{I}_{k}$ has cohomology class equal to the element $-\frac{1}{n_{k}}$. For $l > k$, the projection map $H^{2}(\Gamma_{\ov{F}/F}, u_{l}(\ov{F})) \to H^{2}(\Gamma_{\ov{F}/F}, u_{k}(\ov{F}))$ corresponds (via the $\ov{F}$-fiber functor given by direct limit of the $\ov{F}$-fiber functors for each $R\mc{I}_{k}$, which are evidently compatible) to the multiplication-by-$n_{l}/n_{k}$ map from $\frac{1}{n_{l}}\Z/\Z$ to $\frac{1}{n_{k}}\Z/\Z$. We thus conclude that the category $\RigIsoc_{F}$ corresponds to the class $-1 \in \widehat{\Z} = \varprojlim_{k} \frac{1}{n_{k}}\Z/\Z = H^{2}(\Gamma_{\ov{F}/F}, u(\ov{F}))$, which is exactly what we want. All in all, we have proved:

\begin{thm}
    The Galois gerbe given associated to the gerbe of fiber functors of $\RigIsoc_{F}$ represents the class $-1 \in \widehat{\Z} \cong H^{2}(\Gamma_{\ov{F}/F},u(\ov{F}))$. In particular, we may identify $\RigIsoc_{F}$ with $\Rep(\mc{E}_{\Kal})$.
\end{thm}

\begin{rem}{\label{rem: Kalbbandvariant}}
    As noted in the introduction, we use an alternative presentation of the group scheme $u$ to the presentation $u'$ (a priori a different group) used in \cite{Kalannals}. Nevertheless, the map $u \to u'$ is an isomorphism. It suffices to show this at the level of character modules, where the map is the inclusion
    \begin{equation*}
        \varinjlim_{E/F} (\Q/\Z)[\Gamma_{E/F}]_{0} \to \varinjlim_{E/F} (\Q/\Z)[\Gamma_{E/F}],
    \end{equation*}
    where the direct limit is over all finite Galois extensions. Surjectivity comes from the fact that for a fixed $x \in \Q/\Z[\Gamma_{E/F}]$ which is $m$-torsion, if we take $K/E/F$ finite Galois with $m \mid [K \colon E]$ then the image of $x$ in $\Q/\Z[\Gamma_{K/F}]$ lies in $\Q/\Z[\Gamma_{K/F}]_{0}$. We warn that this result is false when $F = \mathbb{R}$.
\end{rem}

\section{The structure of $\RigIsoc_{F}$}

In this section we study the structure of $\RigIsoc_{F}$ as an abelian tensor category and construct examples of its objects.

\subsection{The Newton map}
Given an object $X = (V,\tilde{\Phi},\mc{D})$ in $R\mc{I}_{k}$, one can attach a homomorphism $f_{X} \in \text{Hom}_{\ov{F}}(u_{\ov{F}}, \GL(\omega_{k}(X))$ as follows: The group $\mathrm{Res}_{E_{k}/F}(\mu_{n_{k}})(\ov{F})$ is canonically identified with the group of automorphisms of the $\ov{F}$-fiber functor described in the previous subsection, which we will denote by $\omega_{k}$, and then $f_{X}$ is defined by sending $x \in \mathrm{Res}_{E_{k}/F}(\mu_{n_{k}})$ to the $\ov{F}$-linear automorphism of $\omega_{k}(X)$ corresponding to $x$.

\begin{lem}
The dimension of $\omega_{k}(X)$ is $\sum_{(a_{\sigma}/n_{k}) \in \prod_{\Gamma_{E_{k}/F}}[0,1)} \textnormal{dim}_{L_{E_{k}}}(V_{(a_{\sigma}/n_{k})})$.    
\end{lem}

\begin{proof}
This follows immediately from the explicit description in \eqref{keyfiberfunctor} of the functor $\omega_{R\mc{I}_{k}}$, noting that all the internal summands in \eqref{keyfiberfunctor} are identified upon taking the quotient by $IR_{E_{k}/F}$.
\end{proof}

It is straightforward to verify that the homomorphisms $f_{X}$ are compatible with the transition maps $R\mc{I}_{k} \to R\mc{I}_{l}$, and we thus can associate to an object $X$ in $\RigIsoc_{F}$ in the image of $R\mc{I}_{k}$ a morphism $f_{X} \in \text{Hom}_{\ov{F}}(u_{\ov{F}}, \GL(\omega_{k}(X)))$.

We say that an object $X$ of $\RigIsoc_{F}$ is \textit{basic} if the homomorphism $f_{X}$ is valued in $Z(\GL(\omega_{k}(X)))$. One can classify the basic objects as follows:

\begin{prop}\label{prop:Newton}
    An object $X$ of $\RigIsoc_{F}$ is basic if and only if all of the indices in $\mathbb{Q}[\Gamma_{\ov{F}/F}]$ of all its nonzero graded pieces have the same image in $\mathbb{Q}/\mathbb{Z}[\Gamma_{\ov{F}/F}]$. 
\end{prop}

\begin{proof}
It suffices to prove this for the finite-level Tannakian category $R\mc{I}_{k}$.

An element $x = (x_{\sigma}) \in  \prod_{\Gamma_{E_{k}/F}} \mu_{n_{k}}(\ov{F}) = \mathrm{Res}_{E_{k}/F}(\mu_{n_{k}})(\ov{F})$ acts on the fiber functor by sending $V_{(a_{\sigma}/n_{k})}$ to itself via multiplication by $\prod_{\sigma}x_{\sigma}^{a_{\sigma}}$. This gives for each graded piece $V_{(a_{\sigma}/n_{k})}$ a morphism from $u_{\ov{F}}$ to $Z(\GL_{V_{(a_{\sigma}/n_{k})}}) = \ov{F}^{\times}$, and the resulting map lies in $Z(\GL(\omega_{k}(X)))$ if and only if these maps agree for all $(a_{\sigma}/n_{k})_{\sigma}$. This last condition holds precisely when all tuples $(a_{\sigma}/n_{k})_{\sigma}$ with $V_{(a_{\sigma}/n_{k})} \neq 0$ are the same modulo $\prod_{\Gamma_{E_{k}/F}}\Z$.
\end{proof}

Denote by $\RigIsoc_{F,n}$ the full subcategory of objects whose image under the canonical fiber functor, temporarily denoted by $\omega$, given above (and thus any fiber functor) is $n$-dimensional; we can also consider the full subcategory $\RigIsoc_{F,n,\text{basic}}$ of basic objects. Recall that if $\mathcal{E}_{\text{Kal}}$ is the gerbe of fiber functors for $\RigIsoc_{F}$, then we have an equivalence of categories
\begin{equation}\label{gerbeequiv}
  \RigIsoc_{F,n} \ \xrightarrow{\sim} Z_{\text{alg}}^{1}(\mathcal{E}_{\text{Kal}}, \mathrm{GL}_{n}),
\end{equation}
and we observe that this functor is compatible with the maps from both objects to $\text{Hom}_{\ov{F}}(u_{\ov{F}},\mathrm{GL}_{n,\ov{F}})$. More explicitly, recall from \cite[Proposition 7.1]{Taibi25} that there is an equivalence of categories from the category of pairs $(R,t)$, where $R \colon \mathcal{E}_{\Kal} \to B\mathrm{GL}_{n}$ is a morphism of fibered categories over $F$ and $t$ is a trivialization of $R(\omega)$, to the category $Z_{\text{alg}}^{1}(\mathcal{E}_{\text{Kal}}, \mathrm{GL}_{n})$. The claimed equivalence \eqref{gerbeequiv} is given by composing this equivalence with the functor sending $X \in \Ob(\RigIsoc_{F,n})$ to the pair consisting of the functor $F \mapsto \underline{\text{Isom}}(F(X), \omega(X))$ and the trivialization of $\underline{\text{Isom}}(\omega(X), \omega(X))$ induced by our distinguished basis for the underlying vector space of $X = (V,\mathcal{D}, \Phi)$.

Note that \eqref{gerbeequiv} induces an equivalence of categories 
\begin{equation*}
  \RigIsoc_{F,n,\text{basic}} \xrightarrow{\sim} Z_{\bas}^{1}(\mathcal{E}_{\text{Kal}}, \mathrm{GL}_{n}).
\end{equation*}

\begin{cor}
    When $X \in \Ob(\RigIsoc_{F})$ is basic, the morphism $f_{X} \in \textnormal{Hom}_{\ov{F}}(u_{\ov{F}}, \GL(\omega_{k}(X)))$ descends to a morphism in $\textnormal{Hom}_{F}(u, \GL(\omega_{k}(X))).$
\end{cor}

This result can also be verified directly using the explicit description of $f_{X}$ given above.

\subsection{Comparison with isocrystals}
We now construct a morphism of Tannakian categories over $F$ from $\Isoc_{F}$ to $\RigIsoc_{F}$. The category $\Isoc_{F,s}$ of isocrystals whose slopes are a multiple of $\frac{1}{s}$ is equivalent to the category $\mc{I}_{F/F, s}$. Indeed, we can define the equivalence by mapping the isocrystal $(V, \Phi)$ to  the space $(V, \tilde{\Phi}, \mc{D})$ where the grading $\mc{D}$ is defined such that $V_{a_{\sigma}/s}$ is the $-a_{\sigma}/s$ slope isotypic piece of $V$ and the $W_{L/F}$-action is given by $\Phi$.

Then for any isocrystal $(V, \Phi)$, we choose some sufficiently large $k \in \N$ such that $(V, \Phi) \in \Isoc_{F,n_k}$ and $(V, \tilde{\Phi}, \mc{D})$ is the corresponding object of $\mc{I}_{F/F, n_k}$. We then apply the quotient functor to obtain an object $R\mc{I}_{F/F, n_k}$ and finally the transition map $R\mc{I}_{F/F, n_k} \to R\mc{I}_{E_k/F, n_k}$.

If we restrict ourselves to isocrystals with slopes of denominator dividing $n_{k}$ for some fixed $k$, then we see that, on bands, the above functor is the composition 
\begin{equation}\label{bandcomp}
\text{Res}_{E_{k}/F}(\mu_{n_{k}}) \xrightarrow{-N_{E_{k}/F}} \mu_{n_{k}} \hookrightarrow \mathbb{G}_{m},
\end{equation}
which, after taking inverse limits, implies that the above morphism of Tannakian categories over $F$ corresponds to the composition $$u \to \mu \to \mathbb{D},$$ where the first map is given by the inverse limit of the negative norm maps. 

Recall that \cite[\S 3.3]{Kaletha18} constructs a homomorphism $\mc{E}_{\Kal}' \to \mc{E}_{\text{iso}}$ for any choice of gerbe $\mc{E}_{\Kal}'$ representing the canonical class $-1 \in H^{2}(F,u)$ which is compatible with the projections to $\Gamma_{\ov{F}/F}$ and agrees with the map \eqref{bandcomp} on bands. Since the choice of $\mc{E}_{\Kal}'$ is only canonical up to conjugation by $u$, any such map is only unique up to post-composing by the automorphisms of $\mc{E}_{\text{iso}}$ induced by $\Int(x)$ for $x \in \mathbb{D}(\ov{F})$. In fact, all maps satisfying these two properties differ in this manner, since $H^{1}(F, \mathbb{D})$ is trivial. In particular, the morphism of gerbes $\mc{E}_{\Kal} \to \mc{E}_{\text{iso}}$ induced by our functor $\Isoc_{F} \to \RigIsoc_{F}$ is such a morphism.

\subsection{Comparison with Fargues' extended isocrystals} 
We first briefly recall the Tannakian category $\text{Isoc}_{F}^{e}$ constructed by Fargues in \cite[\S 6]{Fargues} corresponding to the gerbe $\mc{E}_{\text{Kal}} \times \mc{E}_{\text{iso}}$:

First, one picks a splitting of the short exact sequence
\begin{equation*}
0 \to \varinjlim \prod_{\Gamma_{K/F}} \mathbb{Z} \to \varinjlim \prod_{\Gamma_{K/F}} \Q \to \varinjlim \prod_{\Gamma_{K/F}}, \mathbb{Q}/\mathbb{Z} \to 0,
\end{equation*}
where the direct limits are over all finite Galois extensions $K/F$, with corresponding $2$-cocycle $c$. We will take the splitting to be the unique one valued in $[0,1)$, such that $c$ takes values in $\varinjlim \prod_{\Gamma_{K/F}}\{0,1\}$.

As in \cite[\S 6.2]{Fargues}, one constructs, for a fixed finite Galois $E/F$, a canonical rank one, slope one $E$-isocrystal $\mathbb{L}_{E}$. We will not need the precise definition of this isocrystal (it is built using the Fargues--Fontaine curve)---the only property needed for the construction of $\Isoc^{e}_{F}$ is that for each $K/E$ a finite Galois extension one has that $N_{K/E}(\mathbb{L}_{K}) := \bigotimes_{\tau \in \Gamma_{K/E}} \mathbb{L}_{K} \otimes_{L_{K}, \tilde{\tau}} L_{K}$, where $\tilde{\tau}$ is some lift of $\tau$ to $\Gamma_{L_{K}/E}$, equals $\mathbb{L}_{E}$. Our norm-compatible system of (anti-)uniformizers evidently furnishes us with such a system by setting $\mathbb{L}_{E}$ to the vector space $L_{E}$ with the Frobenius of $E$ acting on $1$ by $\varpi_{E}^{-1}$. 

For a fixed function $\Gamma_{E/F} \xrightarrow{\alpha} \mathbb{Z}$, we define the $E$-isocrystal
\begin{equation*}
\mathbb{L}_{\infty}^{\alpha} := \bigotimes_{\tau \in \Gamma_{E/F}} \mathbb{L}_{E}^{\otimes\alpha(\tau)} \otimes_{L_{E},\tilde{\tau}} L_{E},
\end{equation*}
where $\tilde{\tau}$ is some lift of $\tau$ to $\Gamma_{L_{E}/F}$. By construction, this is a finite dimensional discrete, semilinear $W_{E}$-module over $\ov{L}$.

The Tannakian category $\text{Isoc}^{e}$ has objects given by $\varinjlim \prod_{\Gamma_{K/F}}, \mathbb{Q}/\mathbb{Z}$-graded finite-dimensional $\ov{L}$-vector spaces 
\begin{equation*}
D = \bigoplus_{\alpha \in  \varinjlim \prod_{\Gamma_{K/F}}, \mathbb{Q}/\mathbb{Z}} D_{\alpha}
\end{equation*}
equipped with a discrete, semilinear $W_{F}$-action such that acting by $\tau \in W_{F}$ is an isomorphism from $D_{\alpha}$ to $D_{\prescript{\tau}{}\alpha}$, with tensor product given by 
\begin{equation*}
(D \otimes D')_{\alpha} = \bigoplus_{\beta + \gamma = \alpha} D_{\beta} \otimes_{\ov{F}} D'_{\gamma} \otimes_{\ov{L}} \mathbb{L}_{\infty}^{c(\beta,\gamma)},
\end{equation*}
where $W_{F}$ acts on $\mathbb{L}_{\infty}^{c(\beta,\gamma)}$ by mapping it to $(\mathbb{L}_{\infty}^{c(\beta,\gamma)})^{\tau} \simeq \mathbb{L}_{\infty}^{\prescript{\tau}{}c(\beta,\gamma)}$ via the map $ \mathbb{L}_{\infty}^{\alpha} \to \mathbb{L}_{\infty}^{\alpha} \otimes_{L_{E},\sigma} \ov{L}$, $x \mapsto x \otimes 1$. We will denote all of the above data by $D = (V_{D}, \Phi_{D},\mc{D}_{D},)$ (the underlying vector space, the $W_{F}$-action, and the grading, respectively). The morphisms in $\text{Isoc}^{e}$ are morphisms of graded $W_{F}$-modules.

The goal is to define a morphism of Tannakian categories 
\begin{equation*}
   \RigIsoc_{F} \to \text{Isoc}_{F}^{e}
\end{equation*}
corresponding (possibly up to twisting by a coboundary) to the projection map $\mc{E}_{\text{Kal}} \times \mc{E}_{\text{iso}} \to \mc{E}_{\text{Kal}}$.

At the level of objects, this functor is the exact functor given by sending $(V,\tilde{\Phi},\mc{D})$ to $(V_{[0,1)},\tilde{\Phi},\ov{\mc{D}})$, where $V_{[0,1)}$ denotes the direct sum of the graded pieces of $V$ all of whose entries lie in $\varinjlim \prod_{\Gamma_{K/F}} [0,1)$, which is preserved by the $W_{F}$-action $\tilde{\Phi}$, and $\ov{\mc{D}}$ is the $\varinjlim \prod_{\Gamma_{K/F}}, \mathbb{Q}/\mathbb{Z}$ grading induced by the $\varinjlim \prod_{\Gamma_{K/F}} \mathbb{Q}$-grading. On morphisms the functor sends $f$ to $\ov{f}$ (this makes sense because $f$ is $R_{\ov{F}/F}$-equivariant).

It remains to justify that this functor preserves the tensor structure. To check tensor compatibility we first compute that (for fixed $E/F$ finite Galois) and $(a_{\sigma})_{\sigma}$, $(b_{\sigma})_{\sigma}$ in $\prod_{\Gamma_{E/F}} [0,1)$, denoting by $\tilde{\otimes}$ the tensor product in $\text{Isoc}^{e}$:
\begin{equation*}
V_{(a_{\sigma})} \tilde{\otimes} W_{(b_{\sigma})} = (V_{(a_{\sigma})} \otimes W_{(b_{\sigma})}) \otimes \mathbb{L}_{\infty}^{c(\ov{(a_{\sigma})}, \ov{(b_{\sigma})})} \subseteq (V \otimes W)_{(a_{\sigma} + b_{\sigma})}
\end{equation*}
where $\mathbb{L}^{c(\ov{(a_{\sigma})}, \ov{(b_{\sigma})})}_{\infty}$
is the tensor product of each $\mathbb{L}_{E} \otimes_{L_{E}, \tilde{\sigma}} L_{E}$ such that $a_{\sigma} + b_{\sigma} \geq 1$. On the other hand, first note that for any graded component $U_{(c_{\sigma})}$ that
\begin{equation*}
U_{(c_{\tau})} \otimes (\mathbb{L}_{E} \otimes_{L_{E}, \tilde{\sigma}} L_{E}) = 1_{\delta_{\sigma}(-1)} \cdot U_{(c_{\tau})},
\end{equation*}
where $\delta_{\sigma}(-1)$ denotes the element of $\prod_{\Gamma_{E/F}}\Z$ with $-1$ in the $\sigma$-coordinate and zeros elsewhere such that  $1_{\delta_{\sigma}(-1)} \in R_{E/F}$. Using the above two computations, we find (using that the tensor product $V \otimes W$ in $\RigIsoc_{F}$ is a tensor product of $R_{E/F}$-modules) that
\begin{equation*}
    ([1_{(-c(\ov{(a_{\sigma})}, \ov{(b_{\sigma})}))_{\sigma}} \cdot V_{(a_{\sigma})}] \otimes W_{(b_{\sigma})}) = (V_{(a_{\sigma})} \otimes [1_{(-c(\ov{(a_{\sigma})}, \ov{(b_{\sigma})}))_{\sigma}} \cdot W_{(b_{\sigma})}])  \subseteq (V \otimes W)_{[0,1)}
\end{equation*}
and equals $V_{(a_{\sigma})} \tilde{\otimes} W_{(b_{\sigma})}$. From here, the desired tensor compatibility follows easily. 
 
We conclude this subsection by studying the map on gerbes $\mc{E}_{\Kal} \times \mc{E}_{\text{iso}} \to \mc{E}_{\Kal}$ corresponding to this morphism of Tannakian categories $\RigIsoc_{F} \to \Isoc_{F}^{e}$.

A fiber functor of $\Isoc_{F}^{e}$ is given, at a fixed finite $K/F$-level, by the following composition \footnote{We warn the reader that the first functor is not a tensor functor, but when composed with the second functor it is.}
\begin{equation*}
    \Isoc_{F,K}^{e} \to \Vect_{\ov{F}} \to \text{$R_{K/F}$-mod} \to \Vect_{\ov{F}},
\end{equation*}
where $\Isoc_{F,K}^{e}$ denotes the full subcategory of $\Isoc_{F}^{e}$ where the grading factors through $\prod_{\Gamma_{K/F}} \Q/\Z$, the first functor forgets the $\varinjlim \prod_{\Gamma_{K/F}} \Q/\Z$-grading and $W_{F}$-action, the second functor sends $Y$ to $Y \otimes_{L_{K}} R_{K/F}$, and the third functor is the fiber functor on finitely-generated $R_{K/F}$-modules from \S \ref{sec:mod}. We deduce that the automorphism group this fiber functor, identified with $u_{\ov{F}} \times \mathbb{D}_{\ov{F}}$, acts on the image of the $\prod_{\Gamma_{K/F}} \Q/\Z$-graded piece $V_{(c_{\sigma}/s)_{\sigma}}$ of an object $X$ of $\Isoc_{F}^{e}$ according to the $\prod_{\Gamma_{K/F}} \Q/\Z$-grading via the $u_{\ov{F}}$-factor, showing that, at the level of bands, the morphism of gerbes $\mc{E}_{\Kal} \times \mc{E}_{\text{iso}} \to \mc{E}_{\Kal}$ is the projection map. This lets us deduce that this morphism of gerbes agrees with the projection map $\mc{E}_{\Kal} \times \mc{E}_{\text{iso}} \to \mc{E}_{\Kal}$ up to possibly composing with $\Int(a)$ for $a \in u(\ov{F}) \times \mathbb{D}(\ov{F})$---we suspect that this in fact equals the projection map, but have not verified this. Note that the equality up to $\Int(a)$ already implies that the pullback map $H^{1}_{\text{\'{e}t}}(\mc{E}_{\Kal}, G) \to H^{1}_{\text{\'{e}t}}(\mc{E}_{\Kal} \times \mc{E}_{\text{iso}},G)$ agrees with pullback by the projection map.

\subsection{Reduction to Newton centralizers}
We start with a basic structural result about $\RigIsoc_{F}$:
\begin{lem}
    The category $\RigIsoc_{F}$ is semisimple.
\end{lem}
\begin{proof}
    We work with the category $\text{Rep}(\mc{E}_{\text{Kal}})$ instead, which we showed in \S \ref{sec:RigIsocdef} can be identified with $\RigIsoc_{F}$. By \cite[Proposition 2.23]{DM82}, it suffices to show that $\mc{E}_{\text{Kal}}$ is pro-reductive (in this proof by ``reductive'' we mean that the connected component of the identity is reductive). This follows from the fact that it can be written as a projective limit of groups of the form 
    \begin{equation*}
        1 \to [\frac{\mathrm{Res}_{E/F}(\mu_{n})}{\mu_{n}}](\ov{F}) \to \mc{E}_{E/F,n} \to \Gamma_{K_{E}/F} \to 1,
    \end{equation*}
    where $K_{E}/E/F$ are finite Galois extensions, cf. \cite[\S 3.1]{Kalannals}, which are all finite.
\end{proof}
The previous lemma implies that, in order to construct all objects in $\RigIsoc_{F}$, the important task is to construct its simple objects. To facilitate this, we recall some standard facts about twisted Levi subgroups of $\GL_n$. 

Recall that for a reductive group $G$ defined over $F$, a subgroup $M \subset G$ is a twisted Levi subgroup if it is defined over $F$ and $M_{\ov{F}} \subset G_{\ov{F}}$ is a Levi subgroup. We say that a twisted Levi $M \subset G$ subgroup is \emph{elliptic} if the $F$-split rank of $Z(M)$ equals that of $Z(G)$. We denote the maximal split sub-torus of $Z(M)$ by $A_M$. Suppose that $M \subset G$ is a non-elliptic twisted Levi subgroup. Then $L=Z_{G}(A_M)$ is a proper Levi subgroup of $G$ containing $M$. In particular, the elliptic twisted Levi subgroups of $G$ are precisely those that are not contained in a proper Levi subgroup.

\begin{lem}{\label{lem: GLntwistedlevis}}
    The twisted Levi subgroups of $\GL_n$ over $F$ are isomorphic to
    \begin{equation*}
        \prod\limits_{i=1}^k \Res_{E_i/F} \GL_{n_i},
    \end{equation*}
    where $E_i/F$ is a field extension of degree $m_i$ and $\sum\limits_{i=1}^k m_in_i=n$.
\end{lem}
\begin{proof}
  This statement is well-known. See the proof given in \cite[Lemma 4.3.4]{FintzenSchwein}
\end{proof}
Suppose that $X \in \Ob(\RigIsoc_{F})$, which maps via \eqref{gerbeequiv} to an algebraic $1$-cocycle $z_X: \mc{E}_{\Kal} \rightarrow \GL_n$ with $\nu_X$ the associated algebraic homomorphism of the band.

\begin{lem}\label{lemma:reg}
We continue with $G=\GL_n$. One can conjugate $z_X$ by $G(\ov{F})$ such that $\nu_X$ has trivial Galois action (i.e., it gives an $F$-rational homomorphism $u \to \GL_{n}$). 
\end{lem}
\begin{proof}
     We first argue that for all twisted Levis subgroups $M$, the map 
     \begin{equation*}
     H^1(F, N_G(M)) \rightarrow H^1(F, N_G(M)/M)
     \end{equation*}
     is surjective. Note that for $M$ a maximal torus, this surjectivity is a corollary of Raghunathan's theorem. For $G=\GL_n$ we can prove the result for a general twisted Levi. First, we assume $M$ is a Levi subgroup, and then  $N_G(M) \rightarrow N_G(M)/M$ has a natural $\Gamma_{\ov{F}/F}$-equivariant splitting coming from the permutation matrices. This implies the surjectivity for $M$, and then we prove the general case by twisting.

    Let $M= Z_{G}(\nu_X)$. Since $G=\GL_n$, we have that $M$ is a twisted Levi of $G_E$, where $E$ is the field of definition of $\nu_X$. Notice that the action on $z_X$ by coboundaries induces a conjugation action on $\nu_X$ and hence $M$. The group $M$ might not be $\Gamma_{\ov{F}/F}$-invariant, but since $G = \GL_{n}$ is split, every Levi subgroup over $\ov{F}$ is conjugate to the base-change of an $F$-rational Levi subgroup and we can thus conjugate $\nu_X$ (and twist $z_{X}$ by the corresponding coboundary) such that $M$ is the base-change to $\ov{F}$ of such a subgroup $L$ of $G$.

    Thus we may as well assume $M = L_{\ov{F}}$ is $\Gamma_{\ov{F}/F}$-stable.  Now, $z_X$ gives a cocycle in $Z^1(F, N_G(L)/L)$ via the function
    \begin{equation*}
        \sigma \mapsto \bar{\sigma} \mapsto z_{X}(\bar{\sigma}) \hspace{1mm} \text{mod} \hspace{1mm} L(\ov{F}),
    \end{equation*}
    where $\bar{\sigma}$ is any lift of $\sigma \in \Gamma_{\ov{F}/F}$ to $\mc{E}_{\Kal}$ (the cocycle is independent of the lift).
   By our above argument, we may modify $z_X$ by a coboundary such that this cocycle comes via projection from some $z \in Z^1(F, N_G(L))$. Now, suppose there exists some $h \in G(\ov{F})$ such that $h\nu_Xh^{-1}$ is $\Gamma_{\ov{F}/F}$-fixed. Then  we compute this is equivalent to $h^{-1}\gamma(h)z^{-1}_{\gamma} \in L(\ov{F})$ for each $\gamma \in \Gamma_{\ov{F}/F}$. This would certainly follow if we knew that $z_{\gamma}$ came from a class in $H^1(F, N_G(L))$ that is trivial in $H^1(F,G)$. But this latter group is trivial again because $G = \GL_{n}$.
\end{proof}
From now on, we assume $\nu_X$ is fixed by $\Gamma_{\ov{F}/F}$.

\begin{rem}
    The cohomology set $H_{\text{alg}}^{1}(\mc{E}_{\Kal}, G)$ for a connected reductive group $G$ is studied in detail in \cite{DS24}, which excusively works with so-called ``regular'' classes, which are those arising as the image of $H^{1}_{\text{alg}}(\mc{E}_{\Kal}, T)$ for a maximal torus $T$ of $G$. Lemma \ref{lemma:reg} implies that every class in $[z] \in H_{\text{alg}}^{1}(\mc{E}_{\Kal}, \GL_{n})$ is regular, as follows: One combines \cite[Lemma 2.8]{DS24} which says, when $z|_{u}$ is defined over $F$, we have that $z \in Z_{\bas}^{1}(\mc{E}_{\Kal}, Z_{G}(z|_{u}))$ with the fact (\cite[Corollary 3.7]{Kalannals}) that $H_{\text{alg}}^{1}(\mc{E}_{\Kal}, T) \to H_{\bas}^{1}(\mc{E}_{\Kal}, M)$ is surjective for any connected reductive group $M$ with elliptic maximal torus $T$.
\end{rem}

\begin{lem}\label{lem:simple}
Under the assumption above, suppose $M \subset \GL_n$ is a minimal twisted Levi subgroup through which $z_X$ factors. If $X$ is simple, then $z_X$ is basic in $M$ (in other words, $\nu_X$ factors through $Z(M)$) and $M$ is elliptic. On the other hand, $X$ is simple if and only if $z_X$ factors through only elliptic twisted Levi subgroups of $\GL_{n}$.
\end{lem}
\begin{proof}
Suppose $z_X$ is not basic in $M$. Then let $M' \subset M$ be the centralizer of $\nu_X$ in $M$. Since $G=\GL_n$, we claim that $M'$ is a twisted Levi subgroup of $M$ and hence $G$. Indeed, $M'$ is the intersection of centralizers of commuting semisimple elements of $\GL_n$, each of which is a twisted Levi subgroup. Moreover, $M' \subset M$ is proper since $\nu_X$ is not basic. The cocycle $z_X$ factors through $M'$ since $x \in \mc{E}_{\Kal}$ lifting $\sigma \in \Gamma_{\ov{F}/F}$ satisfies $\Int(x)(\nu_X) = \Int(x)(\sigma(\nu_X))=\nu_X$, hence $M$ is not minimal. 

If $X$ is simple, then $z_X$ cannot factor through any non-elliptic twisted Levi subgroup, since by the above discussion this would imply $z_X$ factors through a Levi subgroup of $G$ and hence is not simple.

For the second part, observe that $X$ is simple if and only if $z_X$ does not factor through a proper Levi subgroup of $G$ and by the discussion before Lemma \ref{lem: GLntwistedlevis}, this is equivalent to $z_X$ not factoring through any non-elliptic twisted twisted Levi subgroup. 
\end{proof}

\subsection{Classification of simple objects}
By Lemma \ref{lem:simple} it suffices to understand the $X \in \Ob(\RigIsoc_{F})$ corresponding to basic classes in the elliptic twisted Levi subgroups $M \xrightarrow{\sim} \Res_{E/F}(\GL_{s})$ of $\GL_n$ which are not isomorphic to an object of the form $Y^{\oplus c}$ for a basic object $Y$ corresponding to $\Res_{E/F}(\GL_{s'})$ for $cs' = s$ with $c>1$. Indeed, such an $X$ would have $z_{X}$ valued in the non-elliptic twisted Levi subgroup $\Res_{E/F}(\GL_{s'})^{c}$.

By \cite[Corollary 5.4]{Kalannals}, for a twisted Levi subgroup $M \subset \GL_n$, we have a bijection
\begin{equation}\label{TNM}
    H^1_{\bas}(\mc{E}_{\Kal}, M) \cong X^{\ast}(\pi_0(Z(\wh{\bar{M}})^+))=X^{\ast}(Z(\wh{\bar{M}})^+)_{\tor} ,
\end{equation}
where $Z(\wh{\bar{M}})^+$ is the pre-image of $Z(\wh{M})^{\Gamma_{\ov{F}/F}}$ in $\wh{\bar{M}} \cong \wh{M}_{\msc} \times \varprojlim\limits_{z \mapsto z^n} Z(\wh{M})^{\circ}$.
If $M$ is elliptic then 
\begin{equation*}
    M \cong \Res_{E/F} \GL_{s},
\end{equation*}
for some extension $E/F$ of degree $m$ such that $n=ms$.
Then $Z(\wh{M})^{\Gamma_{\ov{F}/F}} \cong \C^{\times}$ and $Z(\wh{\bar{M}})^+$ is isomorphic to the subgroup of $(\mu_s)^m \times (\varprojlim\limits_{z \mapsto z^n} \C^{\times})^m$ given by $((a_1,\dots,a_m), (b_1, \dots , b_m))$ such that $a_1\ov{b_1}= \dots =a_m\ov{b_m}$ where $\ov{b_i}$ is the projection of $b_i$ to the final $\C^{\times}$. Then we have the Cartesian diagram
\begin{equation*}
    \begin{tikzcd}
Z(\wh{\bar{M}})^+ \arrow[r] \arrow[d] & (\mu_s)^m \times \varprojlim\limits_{z \mapsto z^n} Z(\wh{M}) \arrow[d]  \\
Z(\wh{M})^{\Gamma_{\ov{F}/{F}}} \arrow[r] & Z(\wh{M})
    \end{tikzcd}
\end{equation*}
Passing to characters, we have the identifications
\begin{equation*}
    X^{\ast}(Z(\wh{M})^{\Gamma_{\ov{F}/F}}) = X^{\ast}(Z(\wh{M}))_{\Gamma_{\ov{F}/F}} = X^{\ast}(Z(\wh{M}))/ X^{\ast}(Z(\wh{M}))_0,
\end{equation*}
where $X^{\ast}(Z(\wh{M}))_0$ is generated by $\gamma(x) - x$ for $\gamma \in \Gamma_{\ov{F}/F}$ and $x \in X^{\ast}(Z(\wh{M}))$ or equivalently the augmentation ideal under identification of $X^{\ast}(Z(\wh{M}))$ with $\Z[\Gamma_{E/F}]$ (or $\Z[\Gamma_{\ov{F}/E} \backslash \Gamma_{\ov{F}/F}] = \Hom_{F}(E,\ov{F})$ if $E/F$ is not Galois). Further, note that $X^{\ast}(\varprojlim\limits_{z \mapsto z^n} Z(\wh{M})) = X^{\ast}(Z(\wh{M}))_{\Q}$.

Thus we finally have that
\begin{equation*}
    X^{\ast}(Z(\wh{\bar{M}})^+) \cong [(\Z/s\Z)^m \times X^{\ast}(Z(\wh{M}))_{\Q}] / X^{\ast}(Z(\wh{M}))_0,
\end{equation*}
and hence
\begin{equation*}
      X^{\ast}(Z(\wh{\bar{M}})^+)_{\tor} \cong [(\Z/s\Z)^m \times X^{\ast}(Z(\wh{M}))_{\Q, 0}] / X^{\ast}(Z(\wh{M}))_0.  
\end{equation*}

The inflation-restriction exact sequence  \cite[(3.5)]{Kalannals}
\begin{equation*}
    1 \rightarrow H^1(F, M) \rightarrow H^1_{\bas}(\mc{E}_{\Kal}, M) \rightarrow \Hom(u,Z(M))^{\Gamma_{\ov{F}/F}}, 
\end{equation*}
and vanishing of the first term implies that $H^1_{\bas}(\mc{E}_{\Kal}, M) \cong X^{\ast}(Z(\wh{\bar{M}})^{\Gamma_{\ov{F}/F}})$ embeds into $\Hom(u,Z(M))^{\Gamma_{\ov{F}/F}}$, which admits the follow chain of identifications
\begin{equation*}
\cong \Hom(X^{\ast}(Z(M)), X^{\ast}(u))^{\Gamma_{\ov{F}/F}} = \Hom(X^{\ast}(Z(M)) , \Q/\Z) \cong X^{\ast}(Z(\wh{M}))_{\Q} / X^{\ast}(Z(\wh{M})) \cong (\Q/\Z)^m.
\end{equation*}
We claim that the map
\begin{equation*}
    [(\Z/s\Z)^m \times X^{\ast}(Z(\wh{M}))_{\Q, 0}] / X^{\ast}(Z(\wh{M}))_0 \rightarrow X^{\ast}(Z(\wh{M}))_{\Q} / X^{\ast}(Z(\wh{M}))
\end{equation*}
is given by $(a_i, b_i) \mapsto \frac{b_i}{s} - \frac{a_i}{s}$---it is easy to check the image is $[(\Q^m)_0 + (\Z/s\Z)^m]/\Z^m \subset (\Q/\Z)^m $. 
\begin{proof}[Proof of claim]
To verify that the map has the above claimed form, we check at a finite $k$-level and compose with the intermediate map from the proof of \cite[Proposition 4.21]{DS24}
\begin{equation}\label{intermedeq}
    H^1(\mc{E}_{\Kal}, Z_{M,k} \to M) \to \Hom(X^{*}(Z_{M,k}), \Q/\Z) \to \Hom(X^{*}(Z(M)), \Q/\Z),
\end{equation}
where $Z_{M,k}$ is the subgroup of $Z(M)$ consisting of all those elements whose $k$th power lies in $Z(M_{\text{der}})$. We claim the middle group is canonically identified with the character group of the kernel of the map $\widehat{M/Z_{M,k}} \to \widehat{M}$, denoted by $K$, which is isomorphic to
\begin{equation*}
    \frac{(\Z/s\Z)^{m} \times \Z^{m}}{k\Z^{m}},
\end{equation*}
where the embedding of $k\Z^{m}$ is given by $(a_{1}, \dots, a_{m}) \mapsto (\ov{a_{1}/k}, \dots, \ov{a_{m}/k}) \times (a_{1}, \dots, a_{m})$. Granting this, the first map in \eqref{intermedeq} is the one induced by the inclusion
\begin{equation*}
    (\Z/s\Z)^{m} \times (\Z^{m})_{0} \hookrightarrow (\Z/s\Z)^{m} \times \Z^{m}.
\end{equation*}
Finally, we compute the right-hand map of \eqref{intermedeq} and prove the claim by writing out the explicit identification
\begin{equation*}
    X^{*}(K) \xrightarrow{\sim} \Hom(X^{*}(Z_{M,k}), \Q/\Z).
\end{equation*}
First, set $T$ to be a maximal torus of $M$, and $Z(M_{1}) := Z(M)/Z(M_{\text{der}})$. Then we have the exact sequence
\begin{equation*}
    0 \rightarrow Z_{M,k} \rightarrow T \rightarrow T/Z_{M,k} \rightarrow 0,
\end{equation*}
and isomorphism
\begin{equation*}
    T/Z_{M,k} \cong T_{\ad} \times Z(M_1)/Z(M_1)[k],
\end{equation*}
as in \cite[p. 71]{Kaletha18}. Thus, we have the exact sequence
\begin{equation}\label{annoyingSESbis}
   0 \to X^{*}(T_{\ad}) \oplus X^{*}(\frac{Z(M_{1})}{Z(M_{1})[k]}) \to X^{*}(T) \to  X^{*}(Z_{M,k})  \to 0.
\end{equation}

Now pick an embedding $\wh{T} \hookrightarrow \wh{M}$ and we denote the image also by $\wh{T}$. We now apply the functor $\Hom_{\Z}(-,\Z)$ to \eqref{annoyingSESbis}, omitting the first term, and via the injective resolution $\Z \to \Q \to \Q/\Z$, identifying $\text{Ext}^{1}_{\Z}(X^{*}(Z_{M,k}), \Z)$ with $\Hom(X^{*}(Z_{M,k}),\Q/\Z)$, yields the exact sequence 
\begin{equation}\label{annoyingSES}
X^{*}(\widehat{T}) \to X^{*}(\widehat{T}_{\der}) \oplus X^{*}(Z(\widehat{M})) \to \Hom(X^{*}(Z_{M,k}), \Q/\Z) \to 0,
\end{equation}
where the first map is restriction to $\widehat{T}_{\der}$ plus the restriction along  $Z(\widehat{M}) \xrightarrow{[k]} Z(\widehat{M}) \rightarrow \wh{T}$ (using that $\wh{Z(M_1)} = Z(\wh{M})^{\circ} = Z(\wh{M})$). It suffices to compute the explicit description of the second map in \eqref{annoyingSES}, which maps $X^{*}(K)$ to $\Hom(X^{*}(Z_{M,k}), \Q/\Z)$ and is a connecting homomorphism.

We choose the standard identification of $X^{*}(\widehat{T}) = X^{*}((\mathbb{C}^{\times,s})^{m})$ with $(\Z^{s})^{m}$, so that $X^{*}(\widehat{T}_{\der}) = (\Z^{s}/\Delta(\Z))^{m}$ and $X^{\ast}(Z(\widehat{M})) = (\Z^{s}/(\Z^{s})_{0})^{m}$. We identify the former with $(\Z^{s-1})^{m}$ and the latter with $\Z^{m}$ in the using the maps $((\ov{c_{i,j})_{1 \leq i \leq s}}) \mapsto ((c_{i,j} - c_{s,j})_{1 \leq i \leq s-1})$ and $(\ov{y_{i,j}}) \mapsto (y_{1,j} + \dots + y_{s,j})$, respectively (corresponding to bases $\{e_{1,j}, \dots ,e_{s-1,j}\}_{1 \leq j \leq m}, \{e_{1,j}\}_{1 \leq j \leq m}$). Under this choice of coordinates, the first map explicitly sends $((c_{i,j})_{1 \leq i \leq s})_{1 \leq j \leq m}$ to $(c_{1,j} -c_{s,j}, \dots c_{s-1,j}-c_{s,j},c_{1,j}+ \dots + c_{s,j})_{1 \leq j \leq m}$. Thus, given $(x_{1,j}, \dots x_{s-1,j}, y_{j})_{j} \in X^{*}(\widehat{T}_{\der}) \oplus X^{*}(Z(\widehat{M})),$ we lift it to $X^{*}(\widehat{T})_{\Q}$, yielding $((x_{i,j}+y_{j}/ks - (x_{1,j} + \dots x_{s-1,j})/s, y_{j}/ks - (x_{1,j} + \dots x_{s-1,j})/s)_{1 \leq i \leq s-1})_{j}$ and then take its image in $X^{*}(\widehat{T})_{\Q/\Z}$, yielding 
\begin{equation*}
(y_{j}/ks - (x_{1,j} + \dots x_{s-1,j})/s)_{j} \in (\Z/ks\Z)^{m} = \Hom((\Z/ks\Z)^{m}, \Q/\Z) = \Hom(X^{*}(Z_{M,k}), \Q/\Z).
\end{equation*}

It follows that when we factor the image of $(x_{1,j}, \dots x_{s-1,j}, y_{j})_{j}$ as in the previous paragraph via the second map of \eqref{annoyingSES} through restriction to $X^{*}(Z(\widehat{M}_{\text{der}})) \oplus X^{*}(Z(\widehat{M}))$, we obtain 
\begin{equation*}
(x_{1,j} + \dots + x_{s-1,j}, y_{j})_{j} \in X^{*}(Z(\widehat{M}_{\text{der}})) \oplus X^{*}(Z(\widehat{M})) \mapsto (y_{j}/ks - (x_{1,j} + \dots x_{s-1,j})/s)_{j},
\end{equation*}
which is exactly the claimed map.

\end{proof}

It remains to construct elements $X$ of $\RigIsoc_{F}$ with cocycle $z_X$ factoring through $M \xrightarrow{\sim} \Res_{E/F} \GL_s$ and basic for that group with gradings matching the above slopes. Note that since there is an injection into $\Hom(u,Z(M))^{\Gamma_{\ov{F}/F}}$, we can identify the class of a basic object from just its grading. 

\begin{ex}\label{objex}
We consider the case where $E/F$ is unramified of degree $m$. Suppose the image of the slope map is the class of $(a_{\sigma}/r)_{\sigma}$ for some $r$ divisible by $s$ and $a_{\sigma} \in \Z$ and $\sigma \in \Gamma_{E/F}$ such that $\sum\limits_{\sigma \in \Gamma_{E/F}} a_{\sigma}/r = b/s$ for some $b \in \Z$. 

We first analyze the case where $b=0$. Up to modifying $(a_{\sigma}/r)_{\sigma}$ by an element of $\Z^m$, we may assume that the $a_{\sigma}/r$ satisfy $a_{\sigma} \neq a_{\sigma'}$ for $\sigma \neq \sigma'$. Now we define the $L_E$ vector space $V$ such that its $\gamma \cdot (a_{\sigma})_{\sigma}$-graded piece is $s$-dimensional, for each $\gamma \in \Gamma_{E/F}$. Let $\alpha = \sum\limits_{\sigma} a_{\sigma}/r \in \Z$ and define $\tilde{\Phi}$ such that $\mf{f}$ acts semilinearly, permutes the grading by $\mf{f}$, and scales a basis of the $(a_{\sigma})_{\sigma}$-graded piece by $\varpi^{\alpha}$. Then $\mf{f}_E$ scales a basis of $V$ by $\varpi^{\alpha}$ and hence $\mf{f}^r_E$ scales by $r\alpha = \sum\limits_{\sigma} a_{\sigma}$. To get an object of $R\mc{I}_{E/F, r}$, we then tensor by $R_{E/F}$.

We now move to the case where $b \neq 0$. We let $K/E$ be the splitting field of a degree $r$ totally ramified extension and we will construct an object in $R\mc{I}_{K/F, r}$. We lift the grading $(a_{\sigma}/r)_{\sigma}$ to $K$. Namely we have a tuple $\mu \in (\frac{1}{r} \Z / \Z)^{\Gamma_{K/F}}$ with some multiple $rk$-many copies of each $a_{\sigma}/r$. Then $\Gamma_{K/E}$ permutes the tuple trivially. Observe that now the sum $\alpha$ over all entries of $\mu$ is an integer. We again define $V$ to be an $L_K$ vector space with $\gamma \cdot \mu$-graded piece $s$-dimensional. Define $\tilde{\Phi}$ such that $\mf{f}$ permutes the graded pieces and scales a basis of $\mu$ by $\prod\limits_{\sigma \in \Gamma_{K/F}} \sigma(\varpi_K)^{a_{\sigma}k}$. Then $\mf{f}_E$ scales a basis of $V$ by an appropriate power of $\mf{f}$ applied to $\prod\limits_{\sigma \in \Gamma_{K/F}} \sigma(\varpi_K)^{a_{\sigma}k}$ and $\mf{f}^r_E$ scales by $\prod\limits_{\sigma \in \Gamma_{K/F}} \sigma(\varpi_K)^{a_{\sigma}rk}$. Finally, we tensor with $R_{K/F}$. The case for a general Galois $E/F$ is similar.
\end{ex}

\begin{lem}
    Given $s \in \N$, $E/F$ Galois, and $(a_{i}) \in (\Q/\Z)^{\Gamma_{E/F}}$ with $\sum a_{i} \in \frac{1}{s}\Z/\Z$, the object $X$ constructed in Example \ref{objex} is simple if and only if $\sum a_{i}$ is coprime to $s$.
\end{lem}

\begin{proof}
    If $c$ is the greatest common divisor of $\sum a_{i}$ and $s$ then $\sum a_{i} \in \frac{1}{s/c}\Z/\Z$, and so Example \ref{objex} also yields an object $Y$ with $z_{Y}$ factoring through $\Res_{E/F}(\GL_{s/c})$, and each $\gamma \cdot (a_{\sigma})$-graded summand of $X$ is the $c$-fold sum of the $\gamma \cdot (a_{\sigma})$-graded summand of $Y$, showing that $X$ cannot be simple if $c>1$. Conversely, if $\sum a_{\sigma}$ is coprime to $s$ then each $\gamma \cdot (a_{\sigma})$-graded summand of $X$ is simple as an $E$-isocrystal, and is a fortiori simple in $\RigIsoc_{F}$.
\end{proof}

    Continuing to take $G = \GL_{n}$, for fixed $[z_{X}] \in H^{1}_{\text{alg}}(\mc{E}_{\Kal}, G)$ we proved that one can find a representative $z_{X}$ such that $\nu_{X}$ is defined over $F$ and $Z_{G}(\nu_{X}) = M \xrightarrow{\sim} \Res_{E/F}(\GL_{s})$ with $E/F$ finite Galois and $X$ simple. Let us see how unique such an $M$ arising from $[z_{X}]$ is. The above work shows that the representative $z_{X}$ with $F$-rational $\nu_{X}$  is unique up to translating by a $g$-coboundary with $g \in G(\ov{F})$ such that $dg \in Z^{1}(F, M_{\ad})$, which replaces $M$ by $\Int(g)(M)$ (we call any $M'$ which is conjugate to $M$ by such an element $g$ a stable conjugate of $M$). Since we insist that $M \cong \Res_{E/F}(\GL_{s})$---a quasi-split group---and any other such $M' \cong \Res_{E'/F}(\GL_{s'})$ is an inner form of $M$ via the twisting isomorphism $\Int(g)$---two such $M$ and $M'$ are conjugate by such a $g$ if and only if $s = s'$ and $E = E'$ (see Remark \ref{Galrem} for more details). In particular, $M$ is completely determined, and thus $z_{X}$ is as well up to translating by $g$-coboundaries for $g \in N_{G}(M)(\ov{F})$ mapping into $[N_{G}(M)/M](F)$. 

  \begin{rem}\label{Galrem}  If $E/F$ and $E'/F$ are finite and separable extensions then the twisted Levi subgroups $\Res_{E/F}(\GL_{s}) \xrightarrow{\sim} M$ and $\Res_{E'/F}(\GL_{s'}) \xrightarrow{\sim} M'$ of $\GL_{n}$ are stably conjugate if and only if $s = s'$, $E$, $E'$ are isomorphic as extensions of $F$, which implies they have the same Galois closure $\widetilde{E}$. The proof of \cite[Lemma 4.3.4]{FintzenSchwein} shows that $M$ and $M'$ are constructed by twisting a Levi subgroup $M_{0}$ of $G$ by two cocycles $z_{M}$ and $z_{M'}$ representing a fixed class in $H^{1}(F,W_{0})$, where $W_{0} := N_{G}(M_{0})/M_{0}$. Since $M \xrightarrow{\sim} \Res_{E/F}(\GL_{s})$, the cocycle $z_{M}$ is the homomorphism $\Gamma_{\ov{F}/F} \to S_{[E \colon F]}$ induced by the action of $\Gamma_{\ov{F}/F}$ on $\Hom_{F}(E,\ov{F})$---in particular, it has kernel $\Gamma_{\ov{F}/\widetilde{E}}$, similarly for $z_{M'}$. Since $z_{M}$ is cohomologous to $z_{M'}$, these two homomorphisms are conjugate, meaning that the $\Gamma_{\ov{F}}$-sets $\Hom(E, \ov{F})$ and $\Hom(E', \ov{F})$ are isomorphic. It follows from Grothendieck's Galois theory that $E$ and $E'$ are isomorphic as field extensions as desired. 
  \end{rem}

    For $E/F$ Galois, after making the identification 
    \begin{equation*}
        [N_{G}(M)/M](F) \xrightarrow{\sim} \Gamma_{E/F},
    \end{equation*}
    the observations in the paragraph preceding Remark \ref{Galrem} are reflected by the fact that, for the tuple $(a_{i}) \in (\Q/\Z)^{\Gamma_{E/F}}$ with $\sum a_{i} \in \frac{1}{s}\Z/\Z$ corresponding to $X$, we obtain an isomorphic object if $(a_{i})$ is replaced by any of its $\Gamma_{E/F}$-translates. In fact, another such tuple $(b_{i})$ gives an object isomorphic to $X$ in $\RigIsoc_{F}$ if and only if they are in the same $\Gamma_{E/F}$-orbit. We conclude that:
    
    \begin{thm} \begin{enumerate} 
    \item{For $E/F$ Galois, the isomorphism classes of simple objects $X$ in $\RigIsoc_{F}$ with $Z_{G}(\nu_{X}) \xrightarrow{\sim} \Res_{E/F}(\GL_{s})$ are in bijection with $\Gamma_{E/F}$-orbits of tuples $(a_{i}) \in (\Q/\Z)^{\Gamma_{E/F}}$ with $\sum a_{i} \in \frac{1}{s}\Z/\Z$, $(\sum a_{i},s) =1$ which are not fixed by any subgroup $H \subseteq \Gamma_{E/F}$ with $H \neq \{1\}$.}
    \item{More generally, for $E/F$ separable with Galois closure $\widetilde{E}$, the isomorphism classes of simple objects $X$ in $\RigIsoc_{F}$ with $Z_{G}(\nu_{X}) \xrightarrow{\sim} \Res_{E'/F}(\GL_{s})$ for $E'$ the image of an $F$-linear embedding of $E$ in $\ov{F}$ are in bijection with $\Gamma_{\widetilde{E}/F}$-orbits of pairs $(E', (a_{i}))$ consisting of such a field $E'$ and a tuple $(a_{i}) \in (\Q/\Z)^{\Hom_{F}(E',\widetilde{E})}$ with $\sum a_{i} \in \frac{ 1}{s }\Z/\Z$, $(\sum a_{i},s) = 1$ which are not induced by an element of $(\Q/\Z)^{\Hom_{F}(K,\widetilde{E})}$ for $F/K/E'$ a proper subextension.}
    \end{enumerate}
    In particular, (2) gives a complete classification of the simple objects of $\RigIsoc_{F}$.
    \end{thm}
\begin{proof} We give the proof of (1)---the proof of (2) is similar after applying Remark \ref{Galrem}.

The description of all simple objects whose isomorphism classes are in the image of $H^{1}(\mc{E}_{\Kal}, M)$ and have Newton centralizer exactly $M$ as the above-described tuples in $(\Q/\Z)^{\Gamma_{E/F}}$ is obtained by combining the explicit computation of the map \eqref{TNM} given above with the observation that if $H \subseteq \Gamma_{E/F}$ fixes $(a_{i})$ then the corresponding simple object $X$ has $\Res_{E^{H}/F}(\GL_{s \cdot |H|}) \subseteq Z_{G}(\nu_{X})$. The fact that two such elements have the same image in $H^{1}(\mc{E}_{\Kal}, G)$ if and only if they are in the same $\Gamma_{E/F}$-orbit follows from the proof of \cite[Proposition 3.9]{DS24}.
\end{proof}

\printbibliography

\end{document}